\documentclass{article}%

\usepackage{amsmath,amssymb,amsfonts}%
\usepackage{amscd}%
\usepackage{theorem}%
\usepackage{enumerate}%
\usepackage{dsfont}%
\usepackage{color}%
\usepackage[arrow,matrix,tips]{xy}%
\usepackage{hyperref}%

\theoremstyle{change}%
\newtheorem{defi}{Definition:}[section]%
\newtheorem{theo}[defi]{Theorem:}%
\newtheorem{prop}[defi]{Proposition:}%
\newtheorem{lemma}[defi]{Lemma:}%
\newtheorem{cor}[defi]{Corollary:}%
{\theorembodyfont{\rmfamily} \newtheorem{remark}[defi]{Remark:}}%
{\theorembodyfont{\rmfamily} \newtheorem{example}[defi]{Example:}}%

\newenvironment{proof}
  {{\bf Proof:}}
  {\qquad \hspace*{\fill} $\Box$}%
  
\newcommand{\id}{\operatorname{id}}%
\newcommand{\N}{\mathbb{N}}%
\newcommand{\Z}{\mathbb{Z}}%
\newcommand{\R}{\mathbb{R}}%
\newcommand{\AC}{\mathcal{A}}%
\newcommand{\CC}{\mathcal{C}}%
\newcommand{\DC}{\mathcal{D}}%
\newcommand{\EC}{\mathcal{E}}%
\newcommand{\HC}{\mathcal{H}}%
\newcommand{\UC}{\mathcal{U}}%
\newcommand{\VC}{\mathcal{V}}%
\newcommand{\NC}{\mathcal{N}}%
\newcommand{\LC}{\mathcal{L}}%
\newcommand{\FC}{\mathcal{F}}%
\newcommand{\BC}{\mathcal{B}}%
\newcommand{\PC}{\mathcal{P}}%
\newcommand{\QC}{\mathcal{Q}}%
\newcommand{\RC}{\mathcal{R}}%
\newcommand{\SC}{\mathcal{S}}%
\newcommand{\sep}{\operatorname{sep}}%
\newcommand{\spn}{\operatorname{span}}%
\newcommand{\cov}{\operatorname{cov}}%
\newcommand{\tp}{\operatorname{top}}%
\newcommand{\supp}{\operatorname{supp}}%
\newcommand{\inner}{\operatorname{int}}%
\newcommand{\cl}{\operatorname{cl}}%
\newcommand{\dist}{\operatorname{dist}}%
\newcommand{\Mis}{\operatorname{M}}%
\newcommand{\Lip}{\operatorname{Lip}}%
\newcommand{\eps}{\varepsilon}%
\newcommand{\tm}{\times}%
\newcommand{\mss}{\!\!\!}%
\newcommand{\rmd}{\mathrm{d}}%

\begin{document}

\title{Metric Entropy of Nonautonomous Dynamical Systems}%
\author{Christoph Kawan\footnote{The author was supported by DFG grant Co 124/17-2 within DFG priority program 1305.}\ \footnote{Institut f\"{u}r Mathematik, Universit\"{a}t Augsburg, 86159 Augsburg, Germany;
e-mail: christoph.kawan@math.uni-augsburg.de}}%
\maketitle

\begin{abstract}
We introduce the notion of metric entropy for a nonautonomous dynamical system given by a sequence $(X_n,\mu_n)$ of probability spaces and a sequence of measurable maps $f_n:X_n\rightarrow X_{n+1}$ with $f_n\mu_n=\mu_{n+1}$. This notion generalizes the classical concept of metric entropy established by Kolmogorov and Sinai, and is related via a variational inequality to the topological entropy of nonautonomous systems as defined by Kolyada, Misiurewicz and Snoha. Moreover, it shares several properties with the classical notion of metric entropy. In particular, invariance with respect to appropriately defined isomorphisms, a power rule, and a Rokhlin-type inequality are proved.%
\end{abstract}

\section{Introduction}%

In the theory of dynamical systems, entropy is an invariant which measures the exponential complexity of the orbit structure of a system. Undoubtedly, the most important notions of entropy are metric entropy for measure-theoretic dynamical systems, sometimes also named \emph{Kolmogorov-Sinai entropy} by its inventors, and topological entropy for topological systems (cf.~Kolmogorov \cite{Kol}, Sinai \cite{Sin} and Adler et al.~\cite{AKM}). There exists a huge variety of modifications and generalizations of these two basic notions. However, most of these only apply to systems which are governed by time-invariant dynamical laws, so-called autonomous dynamical systems. In the literature, one basically finds two exceptions. In the theory of random dynamical systems, which are nonautonomous dynamical systems described by measurable skew-products, both notions of entropy, metric and topological, have been defined and extensively studied (see, e.g., \cite{Bog,FSt,Liu,LQi,Zha}). In particular, the classical variational principle which relates the two notions of entropy to each other, has been adapted to their random versions by Bogensch\"{u}tz \cite{Bog}. The second exception is the quantity introduced in Kolyada and Snoha \cite{KSn}, the topological entropy of a nonautonomous system given as a discrete-time deterministic process on a compact topological space. The theory founded in \cite{KSn} has been further developed in \cite{HWZ,HW2,KMS,Mou,OW2,ZCh,ZLX,ZZH} by several authors. In some of these articles, the definition of entropy has been generalized, in particular to continuous-time systems, to systems with noncompact state space, systems with time-dependent state space, and to local processes. Besides that, there have been other independent approaches (see, e.g., \cite{OWi,PMa}), which essentially lead to the same notion. Both of the nonautonomous versions of entropy, random and deterministic, are intimately related to each other but nevertheless, one cannot draw direct conclusions from the well-developed random theory to the deterministic one except for generic statements (saying that something holds for almost every deterministic system in a large class of such systems parametrized by a random parameter).%

The reason why the deterministic nonautonomous theory of entropy is still quite poor-developed in particular lies in the fact that the notion of metric entropy (together with a variational principle) has not yet successfully been established in that theory. To the best of my knowledge, the only approach in this direction can be found in Zhu et al.~\cite{ZLX}. This work shows that one of the obstacles in establishing a reasonable notion of metric entropy which allows for a variational principle lies in the proof of the power rule which relates the entropies of the time-$t$-maps (the powers of the system) to that of the time-one-map. The aim of this paper is to introduce the notion of metric entropy for nonautonomous measure-theoretic dynamical systems together with a formalism which allows for a power rule and at least the easier part of the variational principle.%

We briefly describe the contents of the paper. In Section \ref{sec_prelims}, we recall the notion of topological entropy for a nonautonomous dynamical system as defined in \cite{KMS} by Kolyada, Misiurewicz and Snoha. This notion of entropy generalizes the one in \cite{KSn} by replacing the state space $X$ (a compact metric space) by a whole sequence $X_n$ of such spaces. The process is then given by a sequence of continuous maps $f_n:X_n\rightarrow X_{n+1}$. As in the classical theory, three equivalent characterizations of entropy are available, via open covers, via spanning sets, or via separated sets. However, one crucial point here is that in the open cover definition sequences of open covers for the spaces $X_n$ with Lebesgue numbers bounded away from zero have to be considered. In order to prove the power rule for this entropy, the additional assumption that the sequence $f_n$ be uniformly equicontinuous is necessary.%

In Section \ref{sec_me}, the metric entropy is defined. Here the system is given by a sequence $f_n:X_n\rightarrow X_{n+1}$ of measurable maps between probability spaces $(X_n,\mu_n)$ such that the sequence $\mu_n$ of measures is preserved in the sense that $f_n\mu_n=\mu_{n+1}$. The metric entropy with respect to a sequence of finite measurable partitions of the spaces $X_n$ can be defined in the usual way (with the obvious modifications), and has similar properties as in the autonomous case. Similarly as in the topological situation (the definition of entropy via sequences of covers), one does not get a reasonable quantity by considering all sequences of partitions. One problem is that information about the initial state can be generated merely due to the fact that the partitions in such a sequence become finer very rapidly. Hence, we have to restrict the class of admissible sequences of partitions, which is done in an axiomatic way by requiring some of the properties that are satisfied in the topological setting by the class of all sequences of open covers with Lebesgue numbers bounded away from zero. This leads to the notion of an \emph{admissible class} which enjoys some nice and natural properties. For instance, in the case of an autonomous measure-preserving system, one can consider the smallest admissible class which contains all constant sequences of partitions, which leads to the classical notion of metric entropy. Several properties of the classical metric entropy carry over to its nonautonomous generalization. In particular, we can establish invariance under appropriately defined isomorphisms, an analogue of the Rokhlin inequality, and a power rule.%

In Section \ref{sec_vi}, we prove for equicontinuous systems the inequality between metric and topological entropy which establishes one part of the variational principle. We adapt the arguments of Misiurewicz's elegant proof from \cite{Mis} by defining an appropriate admissible class of sequences of partitions which is designed in such a way that Misiurewicz's arguments can be applied to its members. This class depends on the given invariant sequence of measures. In general, it might be very small, so that our variational inequality would not give any meaningful information. For this reason we establish different stability conditions for invariant sequences of measures which guarantee that the associated Misiurewicz class contains sequences of arbitrarily fine partitions. These stability conditions capture the intuitive idea that the initial measure $\mu_1$ should not be deformed too much by pushing it forwards by the maps $f_1^n = f_n \circ \cdots \circ f_1$, so that such sequences become an appropriate nonautonomous substitute of invariant measures in the autonomous theory. In particular, we show that the expanding systems studied in Ott, Stenlund, and Young \cite{OSY} satisfy such a stability condition.%

\section{Preliminaries}\label{sec_prelims}%

\subsection{Notation}%

By a \emph{nonautonomous dynamical system} (short NDS) we understand a deterministic process $(X_{1,\infty},f_{1,\infty})$, where $X_{1,\infty} = \{X_n\}_{n\geq1}$ is a sequence of sets and $f_n:X_n\rightarrow X_{n+1}$ a sequence of maps. For all integers $k,n\in\N$ we write%
\begin{equation*}
  f_k^0 := \id_{X_k},\quad f_k^n := f_{k+(n-1)} \circ \cdots \circ f_{k+1} \circ f_k,\quad f_k^{-n} := (f_k^n)^{-1}.%
\end{equation*}
The last notation will only be applied to sets. We do not assume that the maps $f_n$ are invertible. The trajectory of a point $x\in X_1$ is the sequence $\{f_1^n(x)\}_{n\geq0}$. By $f_{k,\infty}$ we denote the sequence $\{f_k,f_{k+1},f_{k+2},\ldots\}$ which defines a NDS on $X_{k,\infty} = \{X_k,X_{k+1},X_{k+2},\ldots\}$.%

We consider two categories of systems, metric and topological. In a metric system, the sets $X_n$ are probability spaces and the maps $f_n$ are measure-preserving. That is, each $X_n$ is endowed with a $\sigma$-algebra $\AC_n$ and a probability measure $\mu_n$ such that the maps $f_n$ are measurable and $f_n\mu_n = \mu_{n+1}$ for all $n\geq1$, where $f_n\mu_n$ denotes the push-forward $(f_n\mu_n)(A) = \mu_n(f_n^{-1}(A))$ for all $A\in\AC_{n+1}$. In this case, we call $\mu_{1,\infty} = \{\mu_n\}_{n\geq1}$ an $f_{1,\infty}$-invariant sequence. In a topological system, each $X_n$ is a compact metric space and the maps $f_n$ are continuous.%

If $X$ is a compact topological space and $\UC$ an open cover of $X$, we denote by $\NC(\UC)$ the minimal cardinality of a finite subcover. If $\UC_1,\ldots,\UC_n$ are open covers of $X$, we write $\bigvee_{i=1}^n\UC_i$ for their join, i.e., the open cover consisting of all the intersections $U_{i_1} \cap U_{i_2} \cap \ldots \cap U_{i_n}$ with $U_{i_j}\in\UC_j$.%

In a metric space $(X,\varrho)$, we denote the open ball centered at $x$ with radius $\eps$ by $B(x,\eps)$ or $B(x,\eps;\varrho)$. We write $\dist(x,A)$ for the distance from a point $x$ to a nonempty set $A$, i.e., $\dist(x,A)=\inf_{a\in A}\varrho(x,a)$. The closure, the interior, and the boundary of a set $A$ are denoted by $\cl A$, $\inner A$ and $\partial A$, respectively.%

Recall that the Lebesgue number of an open cover $\UC$ of a compact metric space $X$ is defined as the maximal $\eps>0$ such that every $\eps$-ball in $X$ is contained in one of the members of $\UC$.%

\subsection{Topological Entropy}\label{subsec_te} In this subsection, we recall the notion of entropy for a topological NDS $(X_{1,\infty},f_{1,\infty})$, as defined in Kolyada et al.~\cite{KMS}. As in the classical autonomous theory, three equivalent definitions are available. We denote the metric of $X_k$ by $\varrho_k$ and define on each of the spaces $X_k$ a class of Bowen-metrics by%
\begin{equation*}
  \varrho_{k,n}(x,y) := \max_{0\leq i\leq n-1}\varrho_{k+i}\left(f_k^i(x),f_k^i(y)\right)\qquad (n\in\N).%
\end{equation*}
It is easy to see that $\varrho_{k,n}$ is a metric on $X_k$ which is topologically equivalent to $\varrho_k$. In order to define the topological entropy of $f_{1,\infty}$, we only use the metrics $\varrho_{1,n}$. A subset $E\subset X_1$ is called $(n,\eps)$-separated if any two distinct points $x,y\in E$ satisfy $\varrho_{1,n}(x,y)>\eps$. A set $F\subset X_1$ $(n,\eps)$-spans another set $K\subset X_1$ if for every $x\in K$ there is $y\in F$ with $\varrho_{1,n}(x,y)\leq\eps$. We let $r_{\sep}(n,\eps,f_{1,\infty})$ denote the maximal cardinality of an $(n,\eps)$-separated subset of $X_1$ and $r_{\spn}(n,\eps,f_{1,\infty})$ the minimal cardinality of a set which $(n,\eps)$-spans $X_1$, and we define%
\begin{eqnarray*}
  h_{\sep}(f_{1,\infty}) &:=& \lim_{\eps\searrow0}\limsup_{n\rightarrow\infty}\frac{1}{n}\log r_{\sep}\left(n,\eps,f_{1,\infty}\right),\\
  h_{\spn}(f_{1,\infty}) &:=& \lim_{\eps\searrow0}\limsup_{n\rightarrow\infty}\frac{1}{n}\log r_{\spn}\left(n,\eps,f_{1,\infty}\right).%
\end{eqnarray*}
The corresponding limits in $\eps$ exist, since the quantities $r_{\sep}(n,\eps,f_{1,\infty})$ and $r_{\spn}(n,\eps,f_{1,\infty})$ are monotone (non-increasing) with respect to $\eps$, and this property carries over to their exponential growth rates. Hence, the limits can also be replaced by the corresponding suprema over all $\eps>0$. With the same arguments as in the autonomous case, one shows that the numbers $h_{\sep}(f_{1,\infty})$ and $h_{\spn}(f_{1,\infty})$ actually coincide. We call their common value the \emph{topological entropy of $f_{1,\infty}$}.%

The definition of topological entropy via open covers has to be modified a little bit in order to fit to the nonautonomous case. Consider a sequence $\UC_{1,\infty} = \{\UC_n\}$ such that $\UC_n$ is an open cover of $X_n$ for each $n\geq1$. The entropy of $f_{1,\infty}$ with respect to the sequence $\UC_{1,\infty}$ is then defined as%
\begin{equation*}
  h_{\cov}(f_{1,\infty};\UC_{1,\infty}) := \limsup_{n\rightarrow\infty}\frac{1}{n}\log\NC\left(\bigvee_{i=0}^{n-1}f_1^{-i}\UC_{i+1}\right).%
\end{equation*}
In contrast to the autonomous case, the upper limit cannot be replaced by a limit (see \cite{KSn} for a counterexample). In order to define the topological entropy of $f_{1,\infty}$ one should not take the supremum of $h_{\cov}(f_{1,\infty};\UC_{1,\infty})$ over all sequences of open covers. The problem is that the value of $h_{\cov}(f_{1,\infty};\UC_{1,\infty})$ might become arbitrarily large just by the fact that the maximal diameters of the open sets in the covers $\UC_n$ exponentially converge to zero for $n\rightarrow\infty$. In this case, information about the initial state can be obtained due to finer and finer measurements even if the system has very regular dynamics. To exclude this, we restrict ourselves to sequences of open covers with Lebesgue numbers bounded away from zero. We denote the family of all these sequences by $\LC(X_{1,\infty})$ and define%
\begin{equation*}
  h_{\cov}(f_{1,\infty}) := \sup_{\UC_{1,\infty} \in \LC(X_{1,\infty})}h_{\cov}(f_{1,\infty};\UC_{1,\infty}).%
\end{equation*}
We leave the easy proof that this number coincides with the topological entropy as defined above to the reader. In the rest of the paper, we write $h_{\tp}(f_{1,\infty})$ for the common value of $h_{\sep}(f_{1,\infty})$, $h_{\spn}(f_{1,\infty})$ and $h_{\cov}(f_{1,\infty})$.%

\begin{remark}
Note that the value of $h_{\tp}(f_{1,\infty})$ heavily depends on the metrics $\varrho_k$ in contrast to the classical autonomous situation. However, in many relevant examples, as, e.g., systems defined by time-dependent differential equations, all of these metrics come from a single metric on a possibly compact space. So in this case the dependence on the metrics disappears due to a canonical choice.%
\end{remark}

The topological entropy of an autonomous system given by a map $f$ satisfies the power rule $h_{\tp}(f^k)=k\cdot h_{\tp}(f)$ for all $k\geq1$. In order to formulate an analogue of this property for NDSs, we have to introduce for every $k\geq1$ the $k$-th power system of the NDS $(X_{1,\infty},f_{1,\infty})$. This is the system $(X^{[k]}_{1,\infty},f^{[k]}_{1,\infty})$, where%
\begin{equation*}
  X^{[k]}_{1,\infty} := \left\{X_{(n-1)k+1}\right\}_{n\geq1},\qquad f^{[k]}_{1,\infty} := \left\{f_{(n-1)k+1}^k\right\}_{n\geq1}.%
\end{equation*}

In case that the spaces $X_n$ coincide, the following result can be found in \cite[Lem.~4.2]{KSn}. Since the proof for the general case works analogously, we omit it.%

\begin{prop}\label{prop_powerrule}
For every $k\geq1$ it holds that%
\begin{equation*}
  h_{\tp}\left(f^{[k]}_{1,\infty}\right) \leq k \cdot h_{\tp}\left(f_{1,\infty}\right).%
\end{equation*}
\end{prop}

In general, the converse inequality in the above proposition fails to hold (see \cite{KSn} for a counterexample). However, if we assume that the family $\{f_n\}$ is equicontinuous, equality does hold. Equicontinuity in this context means uniform equicontinuity, i.e., for every $\eps>0$ there exists $\delta>0$ such that $\varrho_n(x,y) < \delta$ for any $x,y\in X_n$, $n\in\N$, implies $\varrho_{n+1}(f_n(x),f_n(y))<\eps$. In \cite[Lem.~4.4]{KSn} this is proved for the case when the spaces $X_n$ all coincide, by using the definition via separated sets. Here we present a different proof using the definition via sequences of open covers, since we want to carry over the arguments later to the proof of the power rule for metric entropy.%

\begin{lemma}\label{lem_topologicalrefinements}
Let $\UC_{1,\infty}\in \LC(X_{1,\infty})$ and assume that $f_{1,\infty}$ is equicontinuous. Then for each $m\geq1$ the sequence $\VC_{1,\infty}$, defined by $\VC_n := \bigvee_{i=0}^{m-1}f_n^{-i}\UC_{n+i}$, is an element of $\LC(X_{1,\infty})$.%
\end{lemma}

\begin{proof}
Let $\eps>0$ be a common lower bound for the Lebesgue numbers of the covers $\UC_n$. Then, for each $n\geq1$, $\eps$ is also a lower bound for the Lebesgue number of $\VC_n$ with respect to the Bowen-metric $\varrho_{n,m}$. This is proved as follows: Let $x\in X_n$ and assume that $\varrho_{n,m}(x,y)<\eps$. Then $f_n^i(y)$ is contained in the ball $B(f_n^i(x),\eps;\varrho_{n+i})$ for $i=0,1,\ldots,m-1$. Since $\eps$ is a lower bound of the Lebesgue number of $\UC_{n+i}$ for all $i$, we find sets $U_i\in\UC_{n+i}$ such that $B(f_n^i(x),\eps;\varrho_{n+i})\subset U_i$ for $i=0,1,\ldots,m-1$, which implies that%
\begin{eqnarray*}
  B(x,\eps;\varrho_{n,m}) &\subset& U_0 \cap f_n^{-1}(U_1) \cap f_n^{-2}(U_2) \cap \ldots \cap f_n^{-(m-1)}(U_{m-1})\\
  &\in& \bigvee_{i=0}^{m-1}f_n^{-i}\UC_{n+i} = \VC_n.%
\end{eqnarray*}
It is easy to see that from equicontinuity of $f_{1,\infty}$ it follows that also the family $\{f_n^i : n\geq1,\ i = 0,1,\ldots,m-1\}$ is equicontinuous. Hence, we can find $\delta>0$ such that $\varrho_n(x,y)<\delta$ implies $\varrho_{n+i}(f_n^i(x),f_n^i(y))<\eps$ for all $n\geq1$ and $i=0,1,\ldots,m-1$. Therefore, every Bowen-ball $B(x,\eps;\varrho_{n,m})$ contains the $\delta$-ball $B(x,\delta;\varrho_n)$, which shows that $\delta$ is a lower bound for the Lebesgue numbers of the covers $\VC_n$.%
\end{proof}

\begin{lemma}\label{lem_timediscgeneral}
Let $\{a_n\}_{n\geq1}$ be a monotonically increasing sequence of real numbers. Then for every $k\geq1$ it holds that%
\begin{equation*}
  \limsup_{n\rightarrow\infty}\frac{a_n}{n} = \limsup_{n\rightarrow\infty}\frac{a_{nk}}{nk}.%
\end{equation*}
\end{lemma}

\begin{proof}
It suffices to prove the inequality ``$\leq$''. To this end, consider an arbitrary sequence $\{n_l\}_{l\geq1}$ of positive integers converging to $\infty$. For every $l\geq1$ there is an $m_l\in\N_0$ with $m_lk\leq n_l \leq (m_l+1)k$, and $m_l\rightarrow\infty$. This implies%
\begin{equation*}
  \frac{1}{n_l}a_{n_l} \leq \frac{1}{m_lk}a_{(m_l+1)k}.%
\end{equation*} 
It follows that%
\begin{equation*}
  \limsup_{l\rightarrow\infty}\frac{a_{(m_l+1)k}}{m_lk} = \limsup_{l\rightarrow\infty}\frac{m_l+1}{m_l}\frac{a_{(m_l+1)k}}{(m_l+1)k} = \limsup_{l\rightarrow\infty}\frac{a_{m_lk}}{m_lk}.%
\end{equation*}
Hence, we conclude that%
\begin{equation*}
  \limsup_{l\rightarrow\infty}\frac{a_{n_l}}{n_l} \leq \limsup_{l\rightarrow\infty}\frac{a_{m_lk}}{m_lk} \leq \limsup_{m\rightarrow\infty}\frac{a_{mk}}{mk},%
\end{equation*}
which yields the desired inequality.%
\end{proof}

\begin{prop}
If the sequence $f_{1,\infty}$ is equicontinuous, then%
\begin{equation}\label{eq_powerform}
  h_{\tp}\left(f^{[k]}_{1,\infty}\right) = k \cdot h_{\tp}\left(f_{1,\infty}\right) \mbox{\quad for all\ } k\geq1.%
\end{equation}
\end{prop}

\begin{proof}
It suffices to prove the inequality ``$\geq$''. To this end, let $\UC_{1,\infty}\in\LC(X_{1,\infty})$. Define a sequence $\VC_{1,\infty} = \{\VC_n\}$ of open covers for $X^{[k]}_{1,\infty}$ as follows:%
\begin{eqnarray*}
  \VC_n &:=& \UC_{(n-1)k+1} \vee f_{(n-1)k+1}^{-1}\UC_{(n-1)k+2} \vee \ldots \vee f_{(n-1)k+1}^{-(k-1)}\UC_{nk}\\
         &=& \bigvee_{j=0}^{k-1}f^{-j}_{(n-1)k+1}\UC_{(n-1)k+1+j}.%
\end{eqnarray*}
Then we find%
\begin{eqnarray*}
  h_{\cov}\left(f^{[k]}_{1,\infty};\VC_{1,\infty}\right) &=& \limsup_{n\rightarrow\infty}\frac{1}{n}\log\NC\left(\bigvee_{i=0}^{n-1}f_1^{-ik}\VC_{i+1}\right)\allowdisplaybreaks\\
                                                         &=& \limsup_{n\rightarrow\infty}\frac{1}{n}\log\NC\left(\bigvee_{i=0}^{n-1}f_1^{-ik}\bigvee_{j=0}^{k-1}f^{-j}_{ik+1}\UC_{ik+1+j}\right)\allowdisplaybreaks\\
                                                         &=& \limsup_{n\rightarrow\infty}\frac{1}{n}\log\NC\left(\bigvee_{i=0}^{n-1}\bigvee_{j=0}^{k-1}f_1^{-(ik+j)}\UC_{(ik+j)+1}\right)\allowdisplaybreaks\\
                                                         &=& k\cdot\limsup_{n\rightarrow\infty}\frac{1}{nk}\log\NC\left(\bigvee_{i=0}^{nk-1}f_1^{-i}\UC_{i+1}\right)\\
                                                         &=& k\cdot h_{\cov}\left(f_{1,\infty};\UC_{1,\infty}\right).%
\end{eqnarray*}
To obtain the last equality we used Lemma \ref{lem_timediscgeneral}. By Lemma \ref{lem_topologicalrefinements}, $\VC_{1,\infty}\in\LC(X_{1,\infty}^{[k]})$, which implies%
\begin{equation*}
  h_{\tp}\left(f^{[k]}_{1,\infty}\right) \geq h_{\cov}\left(f^{[k]}_{1,\infty};\VC_{1,\infty}\right) = k\cdot h_{\cov}\left(f_{1,\infty};\UC_{1,\infty}\right).%
\end{equation*}
Since this holds for every $\UC_{1,\infty}\in\LC(X_{1,\infty})$, the desired inequality follows.%
\end{proof}

\begin{remark}
Next to the classical notion of entropy for continuous maps on compact spaces, the notion of topological entropy introduced above generalizes several other concepts of entropy. Here are three examples:%
\begin{enumerate}
\item[(i)] Topological entropy for uniformly continuous maps on noncompact metric spaces (cf.~Bowen \cite{Bow}): Consider a uniformly continuous map $f:X\rightarrow X$ on a metric space $X$. The topological entropy of $f$ is defined by%
\begin{equation*}
  h_{\tp}(f) := \sup_{K\subset X} \lim_{\eps\searrow0}\limsup_{n\rightarrow\infty}\frac{1}{n}\log r_{\spn}(n,\eps,K),%
\end{equation*}
where the supremum runs over all compact sets $K\subset X$ and $r_{\spn}(n,\eps,K)$ is the minimal cardinality of a set which $(n,\eps)$-spans $K$. Alternatively, one can take maximal $(n,\eps)$-separated subsets of $K$. If we define for each compact set $K\subset X$ a NDS $f_{1,\infty}^{(K)}$ by%
\begin{equation*}
  X_n := f^{n-1}(K),\quad f_n^{(K)} := f|_{X_n}:X_n \rightarrow X_{n+1},%
\end{equation*}
we see that $h_{\tp}(f)$ can be written as%
\begin{equation*}
  h_{\tp}(f) = \sup_{K\subset X} h_{\tp}(f_{1,\infty}^{(K)}).%
\end{equation*}

\item[(ii)] Topological sequence entropy (cf.~Goodman \cite{Goo}): Here the sequence $X_{1,\infty}$ is constant and the sequence $f_n$ is of the form $f_n = f^{k_n}$, where $f:X\rightarrow X$ is a given continuous map and $(k_n)_{n\geq1}$ an increasing sequence of integers.%

\item[(iii)] Topological entropy of random dynamical systems (cf.~Bogensch\"{u}tz \cite{Bog}): Consider a probability space $(\Omega,\FC,P)$ with an ergodic invertible transformation $\vartheta$ on $\Omega$, and a measurable space $(X,\BC)$. A mapping $\varphi:\Z\tm\Omega\tm X \rightarrow X$ such that $(\omega,x)\mapsto \varphi(n,\omega,x)$ is $\FC\otimes\BC$-measurable for all $n\in\Z$ and $\varphi(n+m,\omega,x)=\varphi(n,\vartheta^m\omega,\varphi(m,\omega,x))$ for all $n,m\in\Z$ and $(\omega,x)\in\Omega\tm X$ is called a random dynamical system on $X$ over $\vartheta$. If $X$ is a compact metric space, $\BC$ is the Borel $\sigma$-algebra of $X$, and the maps $\varphi(n,\omega,\cdot)$ are homeomorphisms, one speaks of a topological random dynamical system. If $\UC$ is an open cover of $X$, one defines for every $\omega\in\Omega$%
\begin{equation}\label{eq_rdsentropy}
  h_{\tp}(\varphi;\UC) := \lim_{n\rightarrow\infty}\frac{1}{n}\log\NC\left(\bigvee_{i=0}^{n-1}\varphi(i,\omega)^{-1}\UC\right).%
\end{equation}
From Kingman's subadditive ergodic theorem it follows that this number exists for almost every $\omega\in\Omega$ and is constant almost everywhere. Then one can take this constant value (for each $\UC$) and define the topological entropy of the random dynamical system by taking the supremum over all open covers $\UC$. If we fix one $\omega\in\Omega$ and consider the number \eqref{eq_rdsentropy}, replacing the limit by a $\limsup$, and then take the supremum over all $\UC$, we obtain the topological entropy of the NDS $(X_{1,\infty},f_{1,\infty})$ given by $X_n := X$, $f_n := \varphi(1,\vartheta^{n-1}\omega,\cdot)$.%
\end{enumerate}
\end{remark}

\begin{remark}
It is an interesting fact that not only Bowen's notion of topological entropy for uniformly continuous maps is a special case of the topological entropy for NDSs, but that for an equicontinuous NDS $(X_{1,\infty},f_{1,\infty})$ also the converse statement is true: $h_{\tp}(f_{1,\infty})$ can be regarded as the topological entropy of a uniformly continuous map, restricted to a compact noninvariant set. To see this, let $X$ be the disjoint sum of the spaces $X_n$, i.e.,%
\begin{equation*}
  X := \coprod_{n=1}^{\infty}X_n,\qquad \varrho(x,y) := \left\{\begin{array}{ll} |n-m| & \mbox{if } x \in X_n,\ y \in X_m,\ n\neq m,\\ \varrho_n(x,y) & \mbox{if } x,y\in X_n. \end{array}\right.%
\end{equation*}
Then a uniformly continuous map $f:X\rightarrow X$ is given by putting $f$ equal to $f_n$ on $X_n$, and we have%
\begin{equation*}
  h_{\tp}(f_{1,\infty}) = h_{\tp}(f,X_1).%
\end{equation*}
This observation in particular allows to conclude the power rule from the corresponding power rule for Bowen's entropy. Taking the supremum of $h_{\tp}(f,K)$ over all compact subsets $K$ of $X$ gives the quantity called the \emph{asymptotical topological entropy} of $f_{1,\infty}$ in \cite{KSn}, defined by $\lim_{n\rightarrow\infty} h_{\tp}(f_{n,\infty})$.%
\end{remark}

\section{Metric Entropy}\label{sec_me}%

In this section, we introduce the metric entropy of a NDS.%

\subsection{The Entropy with Respect to a Sequence of Partitions}

Recall that the entropy of a finite measurable partition $\PC = \{P_1,\ldots,P_k\}$ of a probability space $(X,\AC,\mu)$ is defined by%
\begin{equation*}
  H_{\mu}(\PC) := -\sum_{i=1}^k\mu(P_i)\log\mu(P_i),%
\end{equation*}
where $0 \cdot \log 0 := 0$, and satisfies $0 \leq H_{\mu}(\PC) \leq \log k$. The equality $H_{\mu}(\PC)=\log k$ holds iff all members of $\PC$ have the same measure.%

If $\PC$ and $\QC$ are two measurable partitions of $X$, the joint partition $\PC \vee \QC = \{P \cap Q : P\in\PC,\ Q\in\QC\}$ satisfies $H_{\mu}(\PC \vee \QC) \leq H_{\mu}(\PC) + H_{\mu}(\QC)$.%

Now consider a metric NDS $(X_{1,\infty},f_{1,\infty},\mu_{1,\infty})$, where $\mu_{1,\infty}$ denotes the sequence of probability measures with $f_n\mu_n = \mu_{n+1}$.  Let $\PC_{1,\infty} = \{\PC_n\}$ be a sequence such that $\PC_n$ is a finite measurable partition of $X_n$ for every $n\geq1$, and define%
\begin{equation}\label{eq_dynamicalentropyonpartition}
  h(f_{1,\infty};\PC_{1,\infty}) := \limsup_{n\rightarrow\infty}\frac{1}{n}H_{\mu_1}\left(\bigvee_{i=0}^{n-1}f_1^{-i}\PC_{i+1}\right).%
\end{equation}
We call this number the \emph{metric entropy of $f_{1,\infty}$ with respect to $\PC_{1,\infty}$}. Note that in the autonomous case this definition reduces to the usual definition of metric entropy with respect to a partition. In this case, the $\limsup$ is in fact a limit, which follows from a subadditivity argument. However, in the general case considered here, subadditivity does not necessarily hold. (In \cite{KSn}, one finds a counterexample for the topological case, which can be modified to serve as a counterexample in the metric case, since this system preserves the Lebesgue measure.) For an autonomous system given by a map $f$ with an invariant measure $\mu$ and a partition $\PC$, we also use the common notations $h_{\mu}(f;\PC)$ and $h_{\mu}(f) = \sup_{\PC}h_{\mu}(f;\PC)$.%

Several well-known properties of the entropy with respect to a partition carry over to its nonautonomous generalization. In order to formulate these properties, we have to introduce some notation. We say that a sequence $\PC_{1,\infty}$ of measurable partitions is \emph{finer} than another such sequence $\QC_{1,\infty}$ if $\PC_n$ is finer than $\QC_n$ for every $n\geq1$ (i.e., every element of $\PC_n$ is contained in an element of $\QC_n$). In this case, we write $\PC_{1,\infty} \succeq \QC_{1,\infty}$. If $\PC_{1,\infty}$ and $\QC_{1,\infty}$ are two sequences of measurable partitions, we define their join $\PC_{1,\infty}\vee\QC_{1,\infty} := \{\PC_n\vee\QC_n\}_{n\geq1}$. For a sequence $\PC_{1,\infty}$ and $m\geq1$ we define another sequence $\PC^{\langle m \rangle}_{1,\infty}(f_{1,\infty})$ by%
\begin{equation*}
  \bigvee_{i=0}^{m-1}f_1^{-i}\PC_{i+1},\quad \bigvee_{i=0}^{m-1}f_2^{-i}\PC_{i+2},\quad\ldots,\quad \bigvee_{i=0}^{m-1}f_k^{-i}\PC_{i+k},\quad \ldots%
\end{equation*}
Finally, recall the definition of conditional entropy for partitions of a probability space $(X,\AC,\mu)$. If $A,B\in\AC$ with $\mu(B)>0$, then $\mu(A|B) := \mu(A\cap B)/\mu(B)$. If $\PC$ and $\QC$ are two partitions of $X$, the conditional entropy of $\PC$ given $\QC$ is%
\begin{equation*}
  H_{\mu}(\PC|\QC) := -\sum_{Q\in\QC}\mu(Q)\sum_{P\in\PC}\mu(P|Q)\log\mu(P|Q).%
\end{equation*}
Some well-known properties of the conditional entropy are summarized in the following proposition (cf., e.g., Katok and Hasselblatt \cite{KHa}).%

\begin{prop}\label{prop_conditionalentropy}
Let $\PC$, $\QC$ and $\RC$ be partitions of $X$.%
\begin{enumerate}
\item[(i)] $H_{\mu}(\PC|\QC)=0$ iff $\QC$ is finer than $\PC$ (modulo null sets).%
\item[(ii)] $H_{\mu}(\PC\vee\QC|\RC) = H_{\mu}(\PC|\RC) + H_{\mu}(\QC|\PC \vee \RC)$.%
\item[(iii)] If $\RC$ is finer than $\QC$, then $H_{\mu}(\PC|\RC) \leq H_{\mu}(\PC|\QC)$.%
\item[(iv)] $0 \leq H_{\mu}(\PC|\QC) \leq H_{\mu}(\PC)$.%
\item[(v)] $H_{\mu}(\PC|\RC) \leq H_{\mu}(\PC|\QC) + H_{\mu}(\QC|\RC)$.%
\end{enumerate}
\end{prop}

Now we can prove a list of elementary properties of $h(f_{1,\infty};\PC_{1,\infty})$ most of which are straightforward generalizations of the corresponding properties of classical metric entropy.%

\begin{prop}\label{prop_entropyprops}
Let $\PC_{1,\infty}$ and $\QC_{1,\infty}$ be two sequences of finite measurable partitions for $X_{1,\infty}$. Then the following assertions hold:%
\begin{enumerate}
\item[(i)] $0 \leq h(f_{1,\infty};\PC_{1,\infty}) \leq \limsup_{n\rightarrow\infty}(1/n)\sum_{i=1}^n\log\#\PC_i$.%
\item[(ii)] $h(f_{1,\infty};\PC_{1,\infty}\vee\QC_{1,\infty}) \leq h(f_{1,\infty};\PC_{1,\infty}) + h(f_{1,\infty};\QC_{1,\infty})$.%
\item[(iii)] If $\PC_{1,\infty}\succeq\QC_{1,\infty}$, then $h(f_{1,\infty};\PC_{1,\infty})\geq h(f_{1,\infty};\QC_{1,\infty})$.%
\item[(iv)] For every $k\geq1$ it holds that%
\begin{equation*}
  h\left(f_{1,\infty};\PC_{1,\infty}\right) = \limsup_{n\rightarrow\infty}\frac{1}{nk}H_{\mu_1}\left(\bigvee_{i=0}^{nk-1}f_1^{-i}\PC_{i+1}\right).%
\end{equation*}
\item[(v)] For every $m\geq1$ it holds that $h(f_{1,\infty};\PC_{1,\infty}) = h(f_{1,\infty};\PC^{\langle m \rangle}_{1,\infty}(f_{1,\infty}))$.%
\item[(vi)] $h(f_{1,\infty};\PC_{1,\infty}) \leq h(f_{1,\infty};\QC_{1,\infty}) + \limsup_{n\rightarrow\infty}(1/n)\sum_{i=1}^n H_{\mu_i}\left(\PC_i|\QC_i\right)$.%
\item[(vii)] $h(f_{k,\infty};\PC_{k,\infty}) = h(f_{l,\infty};\PC_{l,\infty})$ for all $k,l\in\N$.%
\end{enumerate}
\end{prop}

\begin{proof}
The properties (i)--(iii) follow very easily from the properties of the entropy of a partition. Property (iv) is a consequence of Lemma \ref{lem_timediscgeneral}, since the partitions $\bigvee_{i=0}^{n-1}f_1^{-i}\PC_{i+1}$ become finer with increasing $n$, and hence the sequence $n\mapsto H_{\mu_1}(\bigvee_{i=0}^{n-1}f_1^{-i}\PC_{i+1})$ is monotonically increasing. To show (v), note that for every $n\geq1$ we have the identities%
\begin{eqnarray*}
&&  H_{\mu_1}\left(\bigvee_{i=0}^{n-1}f_1^{-i}\PC^{\langle m \rangle}_{i+1}(f_{1,\infty})\right) = H_{\mu_1}\left(\bigvee_{i=0}^{n-1}f_1^{-i}\bigvee_{j=0}^{m-1}f_{i+1}^{-j}\PC_{j+i+1}\right)\allowdisplaybreaks\\
  &=& H_{\mu_1}\left(\bigvee_{i=0}^{n-1}\bigvee_{j=0}^{m-1} f_1^{-(i+j)}\PC_{i+j+1}\right) = H_{\mu_1}\left(\bigvee_{k=0}^{n+m-2}f_1^{-k}\PC_{k+1}\right).%
\end{eqnarray*}
This implies%
\begin{eqnarray*}
  h\left(f_{1,\infty};\PC^{\langle m \rangle}_{1,\infty}(f_{1,\infty})\right) &=& \limsup_{n\rightarrow\infty}\frac{1}{n}H_{\mu_1}\left(\bigvee_{k=0}^{n+m-2}f_1^{-k}\PC_{k+1}\right)\\
  &=& \limsup_{n\rightarrow\infty}\frac{1}{n}H_{\mu_1}\left(\bigvee_{k=0}^{n-1}f_1^{-k}\PC_{k+1}\right) = h\left(f_{1,\infty};\PC_{1,\infty}\right),%
\end{eqnarray*}
which concludes the proof of (v). Next, let us prove (vi): From Proposition \ref{prop_conditionalentropy} (ii) it follows that%
\begin{eqnarray*}
  H_{\mu_1}\!\mss\mss&&\mss\mss\!\left(\bigvee_{i=0}^{n-1}f_1^{-i}\PC_{i+1}\right) \leq H_{\mu_1}\left(\bigvee_{i=0}^{n-1}f_1^{-i}\PC_{i+1} \vee \bigvee_{i=0}^{n-1}f_1^{-i}\QC_{i+1}\right)\\
                       &=& H_{\mu_1}\left(\bigvee_{i=0}^{n-1}f_1^{-i}\QC_{i+1}\right) + H_{\mu_1}\left(\bigvee_{i=0}^{n-1}f_1^{-i}\PC_{i+1}|\bigvee_{i=0}^{n-1}f_1^{-i}\QC_{i+1}\right).%
\end{eqnarray*}
For the last term in this expression we further obtain%
\begin{eqnarray*}
  H_{\mu_1}\!\mss\mss&&\mss\mss\!\left(\bigvee_{i=0}^{n-1}f_1^{-i}\PC_{i+1}|\bigvee_{i=0}^{n-1}f_1^{-i}\QC_{i+1}\right)\allowdisplaybreaks\\
   &=& H_{\mu_1}\left(\PC_1 \vee f_1^{-1}\bigvee_{i=0}^{n-2}f_2^{-i}\PC_{i+2}|\bigvee_{i=0}^{n-1}f_1^{-i}\QC_{i+1}\right)\allowdisplaybreaks\\
   &=& H_{\mu_1}\left(\PC_1|\bigvee_{i=0}^{n-1}f_1^{-i}\QC_{i+1}\right)
    + H_{\mu_1}\left(f_1^{-1}\bigvee_{i=0}^{n-2}f_2^{-i}\PC_{i+2}|\PC_1 \vee \bigvee_{i=0}^{n-1}f_1^{-i}\QC_{i+1}\right).%
\end{eqnarray*}
Now we use Proposition \ref{prop_conditionalentropy} (iii) to see that this sum can be estimated by%
\begin{equation*}
  \leq H_{\mu_1}\left(\PC_1|\QC_1\right) + H_{\mu_1}\left(f_1^{-1}\bigvee_{i=0}^{n-2}f_2^{-i}\PC_{i+2}|\bigvee_{i=0}^{n-1}f_1^{-i}\QC_{i+1}\right).%
\end{equation*}
Using the same arguments again, for this expression we find%
\begin{eqnarray*}
  &=& H_{\mu_1}\left(\PC_1|\QC_1\right) + H_{\mu_1}\left(f_1^{-1}\PC_2 \vee \bigvee_{i=0}^{n-3}f_1^{-(i+2)}\PC_{i+3}|\bigvee_{i=0}^{n-1}f_1^{-i}\QC_{i+1}\right)\allowdisplaybreaks\\
  &=& H_{\mu_1}\left(\PC_1|\QC_1\right) + H_{\mu_1}\left(f_1^{-1}\PC_2|\bigvee_{i=0}^{n-1}f_1^{-i}\QC_{i+1}\right)\allowdisplaybreaks\\
  && + H_{\mu_1}\left(\bigvee_{i=0}^{n-3}f_1^{-(i+2)}\PC_{i+3}| f_1^{-1}\PC_2 \vee \bigvee_{i=0}^{n-1}f_1^{-i}\QC_{i+1}\right)\allowdisplaybreaks\\
  &\leq& H_{\mu_1}\left(\PC_1|\QC_1\right) + H_{\mu_1}\left(f_1^{-1}\PC_2| f_1^{-1}\QC_2\right)\\
  && + H_{\mu_1}\left(f_1^{-2}\bigvee_{i=0}^{n-3}f_3^{-i}\PC_{i+3}|\bigvee_{i=0}^{n-1}f_1^{-i}\QC_{i+1}\right).%
\end{eqnarray*}
Using $f_1\mu_1 = \mu_2$, we find that $H_{\mu_1}\left(f_1^{-1}\PC_2| f_1^{-1}\QC_2\right) = H_{\mu_2}(\PC_2|\QC_2)$. Going on inductively, we end up with the estimate%
\begin{equation*}
  H_{\mu_1}\left(\bigvee_{i=0}^{n-1}f_1^{-i}\PC_{i+1}|\bigvee_{i=0}^{n-1}f_1^{-i}\QC_{i+1}\right) \leq \sum_{i=1}^n H_{\mu_i}\left(\PC_i|\QC_i\right).%
\end{equation*}
Hence, we obtain%
\begin{equation*}
  h\left(f_{1,\infty};\PC_{1,\infty}\right) \leq h\left(f_{1,\infty};\QC_{1,\infty}\right) + \limsup_{n\rightarrow\infty}\frac{1}{n}\sum_{i=1}^n H_{\mu_i}\left(\PC_i|\QC_i\right),%
\end{equation*}
which finishes the proof of (vi). Finally, we prove (vii): For any $k\in\N$ we find%
\begin{eqnarray*}
  h(f_{k,\infty};\PC_{k,\infty}) &=& \limsup_{n\rightarrow\infty}\frac{1}{n}H_{\mu_k}\left(\PC_k \vee \bigvee_{i=1}^{n-1}f_k^{-i}\PC_{k+i}\right)\allowdisplaybreaks\\
  &\leq& \limsup_{n\rightarrow\infty}\frac{1}{n}\left[H_{\mu_k}\left(\PC_k\right) + H_{\mu_k}\left(\bigvee_{i=1}^{n-1}f_k^{-i}\PC_{k+i}\right)\right]\allowdisplaybreaks\\
  &=& \limsup_{n\rightarrow\infty}\frac{1}{n}H_{\mu_k}\left(f_k^{-1}\bigvee_{i=1}^{n-1}f_{k+1}^{-(i-1)}\PC_{k+i}\right)\allowdisplaybreaks\\
  &=& \limsup_{n\rightarrow\infty}\frac{1}{n}H_{\mu_{k+1}}\left(\bigvee_{i=0}^{n-2}f_{k+1}^{-i}\PC_{(k+1)+i}\right) = h\left(f_{k+1,\infty};\PC_{k+1,\infty}\right).%
\end{eqnarray*}
Using the elementary property of the entropy of partitions that $H(\AC) \geq H(\BC)$ whenever $\AC$ is finer than $\BC$, the converse inequality is proved by%
\begin{eqnarray*}
  h(f_{k,\infty};\PC_{k,\infty}) &=& \limsup_{n\rightarrow\infty}\frac{1}{n}H_{\mu_k}\left(\PC_k \vee \bigvee_{i=1}^{n-1}f_k^{-i}\PC_{k+i}\right)\allowdisplaybreaks\\
                              &\geq& \limsup_{n\rightarrow\infty}\frac{1}{n}H_{\mu_k}\left(\bigvee_{i=1}^{n-1}f_k^{-i}\PC_{k+i}\right)\allowdisplaybreaks\\
                                 &=& \limsup_{n\rightarrow\infty}\frac{1}{n}H_{\mu_k}\left(f_k^{-1}\bigvee_{i=1}^{n-1}f_{k+1}^{-(i-1)}\PC_{k+i}\right)\allowdisplaybreaks\\
                                 &=& \limsup_{n\rightarrow\infty}\frac{1}{n}H_{\mu_{k+1}}\left(\bigvee_{i=0}^{n-2}f_{k+1}^{-i}\PC_{(k+1)+i}\right) = h(f_{k+1,\infty};\PC_{k+1,\infty}).%
\end{eqnarray*}
This implies (vii) and finishes the proof of the proposition.%
\end{proof}

\begin{remark}
Note that the equality in item (vii) of the preceding proposition reveals an essential difference between metric and topological entropy of NDSs, since in the topological setting only the inequality%
\begin{equation*}
  h_{\tp}(f_{k,\infty}) \leq h_{\tp}(f_{k+1,\infty})%
\end{equation*}
holds. A counterexample for the equality is given by a sequence $f_{1,\infty}$ on the unit interval such that $f_1$ is constant and all other $f_n$ are equal to the standard tent map. In this case, clearly $h_{\tp}(f_{1,\infty})=0$, but $h_{\tp}(f_{k,\infty}) = \log 2$ for all $k\geq2$ (see also \cite{KSn} for a counterexample with $h_{\tp}(f_{k,\infty}) < h_{\tp}(f_{k+1,\infty})$ for all $k$). Therefore, the notion of asymptotical topological entropy, as defined in \cite{KSn}, has no meaningful analogue for metric systems.%
\end{remark}

From item (vii) of the preceding proposition we can conclude a similar result as \cite[Thm.~A]{KSn} which asserts that the topological entropy of autonomous systems is commutative in the sense that $h_{\tp}(f\circ g) = h_{\tp}(g\circ f)$.%

\begin{cor}
Consider two probability spaces $(X,\mu)$ and $(Y,\nu)$ and measurable maps $f:X\rightarrow Y$, $g:Y\rightarrow X$ such that $f\mu = \nu$ and $g\nu = \mu$. Then $\mu$ is an invariant measure for $g \circ f$, $\nu$ is an invariant measure for $f\circ g$, and it holds that%
\begin{equation*}
  h_{\nu}(f\circ g) = h_{\mu}(g\circ f).%
\end{equation*}
\end{cor}

\begin{proof}
We consider the NDS $(X_{1,\infty},f_{1,\infty})$ defined by $X_{1,\infty} := \{X,Y,X,Y,\ldots\}$ and $f_{1,\infty} := \{f,g,f,g,\ldots\}$. The corresponding $f_{1,\infty}$-invariant sequence of measures is $\mu_{1,\infty} := \{\mu,\nu,\mu,\nu,\ldots\}$. Consider a finite partition $\QC$ of $Y$ and put $\PC := f^{-1}\QC$. Then, by Proposition \ref{prop_entropyprops} (vii), for $\PC_{1,\infty} := \{\PC,\QC,\PC,\QC,\ldots\}$ we have%
\begin{equation}\label{eq_h12}
  h(f_{1,\infty};\PC_{1,\infty}) \leq h(f_{2,\infty};\PC_{2,\infty}) \leq h(f_{3,\infty};\PC_{3,\infty}) = h(f_{1,\infty};\PC_{1,\infty}).%
\end{equation}
Using Proposition \ref{prop_entropyprops} (iv), we find%
\begin{eqnarray*}
  h\left(f_{1,\infty};\PC_{1,\infty}\right) &=& \limsup_{n\rightarrow\infty}\frac{1}{2n}H_{\mu}\left(\bigvee_{i=0}^{2n-1}f_1^{-i}\PC_{i+1}\right)\\
  &=& \limsup_{n\rightarrow\infty}\frac{1}{2n}H_{\mu}\left(\bigvee_{i=0}^{n-1}f_1^{-2i}\PC_{2i+1} \vee \bigvee_{i=0}^{n-1}f_1^{-(2i+1)}\PC_{2i+2}\right)\\
  &=& \limsup_{n\rightarrow\infty}\frac{1}{2n}H_{\mu}\left(\bigvee_{i=0}^{n-1}(g\circ f)^{-i}\PC \vee \bigvee_{i=0}^{n-1}(g\circ f)^{-i}f^{-1}\QC\right)\\
  &=& \frac{1}{2}\limsup_{n\rightarrow\infty}\frac{1}{n}H_{\mu}\left(\bigvee_{i=0}^{n-1}(g\circ f)^{-i}\PC\right) = \frac{1}{2}h_{\mu}(g\circ f;\PC).%
\end{eqnarray*}
Similarly, we obtain $2h(f_{2,\infty};\PC_{2,\infty}) = h_{\nu}(f\circ g;\QC)$. Hence, from \eqref{eq_h12} we conclude%
\begin{equation*}
  h_{\mu}(g\circ f;\PC) = h_{\nu}(f\circ g;\QC).%
\end{equation*}
Since we can choose $\QC$ freely, this implies $h_{\nu}(f\circ g) \leq h_{\mu}(g\circ f)$. Starting with a partition $\PC$ of $X$ and putting $\QC := g^{-1}\PC$, we get the converse inequality.%
\end{proof}

\begin{remark}
In Balibrea, Jim\'enez L\'opez, and C\'anovas \cite{BJC} one finds proofs for the commutativity of metric and topological entropy which are not based on entropy notions for nonautonomous systems. These commutativity properties were first found in Dana and Montrucchio \cite{DMo}. Later, Kolyada and Snoha \cite{KSn} rediscovered the commutativity of topological entropy.%
\end{remark}

We finish this subsection with an example which shows that the entropy $h(f_{1,\infty};\PC_{1,\infty})$ can be arbitrarily large even for a very trivial system.%

\begin{example}\label{ex_id}
Let $X_{1,\infty}$, $f_{1,\infty}$, and $\mu_{1,\infty}$ be constant sequences given by $X_n = [0,1]$, $f_n = \id_{[0,1]}$ and $\mu_n = \lambda$ (the standard Lebesgue measure). Consider the family $\PC_{1,\infty}$ of partitions given by%
\begin{equation*}
  \PC_n = \left\{[0,1/k^n),[1/k^n,2/k^n),\ldots,[(k^n-1)/k^n,1]\right\}%
\end{equation*}
for a fixed integer $k\geq2$. Then one easily sees that%
\begin{equation*}
  H_{\mu_1}\left(\bigvee_{i=0}^{n-1}f_1^{-i}\PC_{i+1}\right) = H_{\lambda}(\PC_n) = -\sum_{i=1}^{k^n}\frac{1}{k^n}\log\frac{1}{k^n} = \log k^n = n\log k,%
\end{equation*}
which implies $h\left(f_{1,\infty};\PC_{1,\infty}\right) = \log k$.%
\end{example}

From this example one sees that by taking appropriate sequences of partitions, one obtains arbitrarily large values for the entropy of the identity. Here we have the same problem as we had in defining the topological entropy via sequences of open covers. If the resolution becomes finer at exponential speed, one obtains a gain in information which is not due to the dynamics of the system. Hence, in the definition of the metric entropy of $f_{1,\infty}$, we have to exclude such sequences.%

\subsection{Admissible Classes and Metric Entropy of Nonautonomous Systems}%

To define the entropy of the system $(X_{1,\infty},f_{1,\infty},\mu_{1,\infty})$, we have to choose a sufficiently nice subclass $\EC$ from the class of all sequences $\PC_{1,\infty}$. Then the entropy can be defined in the usual way by taking the supremum over all $\PC_{1,\infty}\in\EC$. In view of the definition of topological entropy in terms of sequences of open covers and Example \ref{ex_id} it is clear that taking all sequences of partitions is too much. Since there is no direct analogue to Lebesgue numbers for measurable partitions, we introduce suitable classes of sequences of partitions by axioms which reflect some properties of the family $\LC(f_{1,\infty})$ defined in Section \ref{sec_prelims}.%

\begin{defi}
We call a nonempty class $\EC$ of sequences of finite measurable partitions for $X_{1,\infty}$ \emph{admissible (for $f_{1,\infty}$)} if it satisfies the following axioms:%
\begin{enumerate}
\item[(A)] For every sequence $\PC_{1,\infty}\in\EC$ there is a bound $N\geq1$ on $\#\PC_n$, i.e., $\#\PC_n \leq N$ for all $n\geq1$.%
\item[(B)] If $\PC_{1,\infty}\in\EC$ and $\QC_{1,\infty}$ is a sequence of partitions for $X_{1,\infty}$ with $\PC_{1,\infty} \succeq \QC_{1,\infty}$, then $\QC_{1,\infty}\in\EC$.%
\item[(C)] $\EC$ is closed with respect to successive refinements via the action of $f_{1,\infty}$. That is, if $\PC_{1,\infty}\in\EC$, then for every $m\geq1$ also $\PC^{\langle m \rangle}_{1,\infty}(f_{1,\infty})\in\EC$.%
\end{enumerate}
\end{defi}

From Axiom (A) it follows that the upper bound in Proposition \ref{prop_entropyprops} (i) is always finite. Moreover, by adding sets of measure zero, we can assume that $\#\PC_n$ is constant for every element of $\EC$. Axiom (B) says that with every sequence $\PC_{1,\infty}\in\EC$, also the sequences which are coarser than $\PC_{1,\infty}$ are contained in $\EC$. Axiom (C) will be essential for proving the power rule for metric entropy. It reflects the property of sequences of open covers stated in Lemma \ref{lem_topologicalrefinements}.%

\begin{defi}
If $\EC$ is an admissible class, we define the \emph{metric entropy} of $f_{1,\infty}$ with respect to $\EC$ by%
\begin{equation}\label{eq_defmetricentropy}
  h_{\EC}(f_{1,\infty}) = h_{\EC}(f_{1,\infty};\mu_{1,\infty}) := \sup_{\PC_{1,\infty}\in\EC}h(f_{1,\infty};\PC_{1,\infty}).%
\end{equation}
\end{defi}

\begin{prop}
Given a metric NDS $(X_{1,\infty},f_{1,\infty})$, let $\EC$ be the class of all sequences of partitions for $X_{1,\infty}$ which satisfy Axiom (A). Then $\EC$ is an admissible class. $\EC$ is maximal, i.e., it cannot be extended to a larger admissible class. Therefore, we denote this class by $\EC_{\max}$ or $\EC_{\max}(X_{1,\infty})$.%
\end{prop}

\begin{proof}
It is obvious that $\EC$ cannot be enlarged without violating Axiom (A). Hence, it suffices to prove that $\EC$ satisfies Axioms (B) and (C). If $\PC_{1,\infty}\in\EC$ and $\QC_{1,\infty}$ is a sequence of partitions which is coarser than $\PC_{1,\infty}$, it follows that $\#\QC_n\leq\#\PC_n$ for all $n\geq1$, which implies $\QC_{1,\infty}\in\EC$. Now consider for some $\PC_{1,\infty}\in\EC$ and $m\geq1$ the sequence $\PC^{\langle m\rangle}_{1,\infty}(f_{1,\infty})$. We have%
\begin{equation*}
  \#\left[\bigvee_{i=0}^{m-1}f_n^{-i}\PC_{i+n}\right] \leq \prod_{i=0}^{m-1}\#\left[f_n^{-i}\PC_{i+n}\right] = \prod_{i=0}^{m-1}\#\PC_{i+n} \leq \left(\sup_{i\geq1}\#\PC_i\right)^m.%
\end{equation*}
This implies that $\EC$ satisfies Axiom (C).%
\end{proof}

The following example shows that $\EC_{\max}$ is in general not a useful admissible class.\footnote{This example has been presented to the author by Tomasz Downarowicz.}%

\begin{example}
We show that $h_{\EC_{\max}}(f_{1,\infty}) = \infty$ whenever the maps $f_i$ are bi-measurable and the spaces $(X_n,\mu_n)$ are non-atomic. Indeed, for every $k\geq1$ we find a sequence $\PC_{1,\infty}$ of partitions with $\#\PC_n \equiv k$ such that $h(f_{1,\infty};\PC_{1,\infty})=\log k$, which is constructed as follows. On $X_1$ take a partition $\PC_1$ consisting of $k$ sets with equal measure $1/k$. Then $\QC_2 := f_1 \PC_1$ is a partition of $X_2$ into $k$ sets of equal measure. Partition each element $Q_i$ of $\QC_2$ into $k$ sets $Q_{i1},\ldots,Q_{ik}$ of equal measure $1/k^2$. Then define a new partition $\PC_2$ of $X_2$ consisting of the sets $P^2_1:=Q_{11} \cup Q_{21} \cup \ldots \cup Q_{k1}$, $P^2_2 := Q_{12}\cup\ldots\cup Q_{k2}$, $\ldots$, $P^2_k := Q_{1k}\cup\ldots \cup Q_{kk}$. Also $\PC_2$ is a partition of $X_2$ into $k$ sets of equal measure $1/k$, and $\PC_2$, $\QC_2$ are independent. This implies%
\begin{eqnarray*}
  H_{\mu_1}(\PC_1 \vee f_1^{-1}\PC_2) &=& H_{\mu_1}(f_1^{-1}\QC_2 \vee f_1^{-1}\PC_2)\\
  &=& H_{\mu_2}(\QC_2 \vee \PC_2) = H_{\mu_2}(\QC_2) + H_{\mu_2}(\PC_2) = 2\log k.%
\end{eqnarray*}
Inductively, one can proceed this construction. For $i$ from $1$ to some fixed $n$, assume that $\PC_i$ is a partition of $X_i$ into $k$ sets of equal measure such that $\RC_n:=\PC_1 \vee f_1^{-1}\PC_2 \vee \ldots \vee f_1^{-(n-1)}\PC_n$ consists of $k^n$ sets of equal measure. Then consider the partition $\QC_{n+1} := f_1^n\RC_n$ of $X_{n+1}$. Let $\RC_n = \{R_1,\ldots,R_{k^n}\}$ and partition each $R_i$ into $k$ sets of equal measure $1/k^{n+1}$, say $R_i = R_{i1}\cup\ldots\cup R_{ik}$. Define the partition $\PC_{n+1} = \{P^{n+1}_1,\ldots,P^{n+1}_k\}$ by $P_j^{n+1} := R_{1j}\cup\ldots\cup R_{k^nj}$. This gives%
\begin{eqnarray*}
  H_{\mu_1}\left(\bigvee_{i=0}^n f_1^{-i}\PC_{i+1}\right) &=& H_{\mu_1}\left(\RC_n\vee f_1^{-n}\PC_{n+1}\right) = H_{\mu_1}\left(f_1^{-n}\QC_{n+1} \vee f_1^{-n}\PC_{n+1}\right)\\
  &=& H_{\mu_{n+1}}(\QC_{n+1}\vee \PC_{n+1}) = H_{\mu_{n+1}}(\QC_{n+1}) + H_{\mu_{n+1}}(\PC_{n+1})\\
  &=& \log k^n + \log k = (n+1) \log k,%
\end{eqnarray*}
which implies $h(f_{1,\infty};\PC_{1,\infty}) = \log k$ for the sequence $\PC_{1,\infty}=\{\PC_n\}$ obtained by this construction.%
\end{example}

As this example shows, we have to consider smaller admissible classes. These are provided by the following proposition whose simple proof will be omitted.

\begin{prop}\label{prop_admissibleclassconstructions}
Arbitrary unions and nonempty intersections of admissible classes are again admissible classes. In particular, for every nonempty subset $\FC\subset\EC_{\max}$ there exists a smallest admissible class $\EC(\FC)$ which satisfies $\FC \subset \EC(\FC) \subset \EC_{\max}$ (defined as the intersection of all admissible classes containing $\FC$). We also call $\EC(\FC)$ the admissible class generated by $\FC$.%
\end{prop}

We also have to show that the metric entropy of a NDS indeed generalizes the usual notion of metric entropy for autonomous systems. To this end, we use the following result.%

\begin{prop}\label{prop_generatedclasses}
Let $\FC$ be a nonempty subset of $\EC_{\max}$. Then%
\begin{equation}\label{eq_generatedclasschar}
  \HC(\FC) := \left\{\QC_{1,\infty}\in\EC_{\max}\ |\ \exists \PC_{1,\infty}\in\FC:\ h(f_{1,\infty};\QC_{1,\infty}) \leq h(f_{1,\infty};\PC_{1,\infty})\right\}%
\end{equation}
is an admissible class with $\FC \subset \HC(\FC) \subset \EC_{\max}$. Consequently, $\EC(\FC) \subset \HC(\FC)$ and it holds that%
\begin{equation*}
  h_{\EC(\FC)}(f_{1,\infty}) = h_{\HC(\FC)}(f_{1,\infty}) = \sup_{\PC_{1,\infty}\in\FC}h\left(f_{1,\infty};\PC_{1,\infty}\right).%
\end{equation*}
\end{prop}

\begin{proof}
It is obvious that $\FC \subset \HC(\FC) \subset \EC_{\max}$. Clearly, $\HC(\FC)$ satisfies Axiom (A). It also satisfies Axiom (B), since any sequence $\RC_{1,\infty}$ of partitions coarser than some $\QC_{1,\infty}\in\HC(\FC)$ satisfies $h(f_{1,\infty};\RC_{1,\infty}) \leq h(f_{1,\infty};\QC_{1,\infty}) \leq h(f_{1,\infty};\PC_{1,\infty})$ for some $\PC_{1,\infty}\in\FC$. With the same reasoning and Proposition \ref{prop_entropyprops} (v), we see that $\HC(\FC)$ satisfies Axiom (C) and hence is an admissible class.%
\end{proof}

The preceding proposition shows not only that there exists a multitude of admissible classes, but also that the metric entropy of $f_{1,\infty}$ can be equal to any of the numbers $h(f_{1,\infty};\PC_{1,\infty})$ by taking the one-point set $\FC := \{\PC_{1,\infty}\}$ as a generator for an admissible class. The next corollary immediately follows.%

\begin{cor}
Assume that the sequences $X_{1,\infty}$, $f_{1,\infty}$, $\mu_{1,\infty}$ are constant, i.e., we have an autonomous system $(X,f,\mu)$. Let $\FC$ be the set of all constant sequences of finite measurable partitions of $X$. Then $h_{\EC(\FC)}(f_{1,\infty}) = h_{\mu}(f)$.%
\end{cor}

\subsection{Invariance, Rokhlin Inequality, and Restrictions} 

In order to be a reasonable quantity, the metric entropy of a system $f_{1,\infty}$ should be an invariant with respect to isomorphims. By an isomorphism between sequences $(X_{1,\infty},\mu_{1,\infty})$ and $(Y_{1,\infty},\nu_{1,\infty})$ of probability spaces we understand a sequence $\pi_{1,\infty} = \{\pi_n\}$ of bi-measurable maps $\pi_n:X_n \rightarrow Y_n$ with $\pi_n\mu_n = \nu_n$. Such a sequence is an isomorphism between the systems $f_{1,\infty}$ on $X_{1,\infty}$ and $g_{1,\infty}$ on $Y_{1,\infty}$ if additionally for each $n\geq1$ the diagram%
\begin{equation*}
\begin{CD}
  X_n @>f_n>> X_{n+1}\\
  @V\pi_n VV @VV\pi_{n+1}V\\
  Y_n @>>g_n> Y_{n+1}
\end{CD}
\end{equation*}
commutes. In this case we also say that the systems $f_{1,\infty}$ and $g_{1,\infty}$ are \emph{conjugate}. If the maps $\pi_n$ are only measurable but not necessarily measurably invertible, we say that the systems $f_{1,\infty}$  
and $g_{1,\infty}$ are \emph{semiconjugate}. The sequence $\pi_{1,\infty}$ is then called a \emph{conjugacy} or a \emph{semiconjugacy from $f_{1,\infty}$ to $g_{1,\infty}$}, respectively.%

Given two admissible classes $\EC$ and $\FC$ for $X_{1,\infty}$ and $Y_{1,\infty}$, resp., we also define the notions of $\EC$-$\FC$-isomorphisms and $\EC$-$\FC$-(semi)conjugacies via the condition that $\pi_{1,\infty}$ respects $\EC$ and $\FC$ in the sense that%
\begin{equation*}
  \PC_{1,\infty} = \{\PC_n\}_{n\geq1} \in \FC \qquad \Rightarrow \qquad \{\pi_n^{-1}(\PC_n)\}_{n\geq1} \in \EC.%
\end{equation*}
In the case of an isomorphism or a conjugacy, the implication into the other direction must hold as well.%

\begin{prop}\label{prop_conjugacy}
Let $(X_{1,\infty},f_{1,\infty},\mu_{1,\infty})$ and $(Y_{1,\infty},g_{1,\infty},\nu_{1,\infty})$ be metric NDS with admissible classes $\EC$ and $\FC$, respectively. Let $\pi_{1,\infty}$ be an $\EC$-$\FC$-semiconjugacy from $f_{1,\infty}$ to $g_{1,\infty}$. Then%
\begin{equation*}
  h_{\FC}(g_{1,\infty}) \leq h_{\EC}(f_{1,\infty}).%
\end{equation*}
\end{prop}

\begin{proof}
First note that the semiconjugacy identities $\pi_{n+1} \circ f_n = g_n \circ \pi_n$ imply $g_1^i \circ \pi_1 = \pi_{i+1} \circ f_1^i$ for all $i$. Let $\PC_{1,\infty} = \{\PC_n\}$ be a sequence of finite measurable partitions for $Y_{1,\infty}$. Fix $n\in\N$ and $P_{j_i} \in \PC_i$, $i=1,\ldots,n$. Then we find%
\begin{eqnarray*}
  \nu_1\left(\bigcap_{i=0}^{n-1} g_1^{-i}P_{j_{i+1}}\right) &=& \mu_1\left(\pi_1^{-1}\bigcap_{i=0}^{n-1}g_1^{-i}P_{j_{i+1}}\right) = \mu_1\left(\bigcap_{i=0}^{n-1} (g_1^i \circ \pi_1)^{-1}P_{j_{i+1}}\right)\allowdisplaybreaks\\
  &=& \mu_1\left(\bigcap_{i=0}^{n-1} (\pi_{i+1} \circ f_1^i)^{-1}P_{j_{i+1}}\right) = \mu_1\left(\bigcap_{i=0}^{n-1} f_1^{-i}\pi_{i+1}^{-1}P_{j_{i+1}}\right).%
\end{eqnarray*}
Define $\QC_{1,\infty} = \{\QC_n\}$ by $\QC_n := \{\pi_n^{-1}(P) : P \in \PC_n\}$ for all $n\geq1$. Then $\QC_n$ is a finite measurable partition of $X_n$ and from the preceding computation we get%
\begin{equation*}
  H_{\nu_1}\left(\bigvee_{i=0}^{n-1} g_1^{-i}\PC_{i+1}\right) = H_{\mu_1}\left(\bigvee_{i=0}^{n-1} f_1^{-i}\QC_{i+1}\right).%
\end{equation*}
Hence, $h(f_{1,\infty};\QC_{1,\infty}) = h(g_{1,\infty};\PC_{1,\infty})$. Writing $\QC_{1,\infty} = \pi_{1,\infty}^{-1}(\PC_{1,\infty})$, we find%
\begin{eqnarray*}
  h_{\FC}(g_{1,\infty}) &=& \sup_{\PC_{1,\infty}\in\FC}h(g_{1,\infty};\PC_{1,\infty}) = \sup_{\PC_{1,\infty}\in\FC}h(f_{1,\infty};\pi_{1,\infty}^{-1}(\PC_{1,\infty}))\\
  &\leq& \sup_{\QC_{1,\infty}\in\EC}h(f_{1,\infty};\QC_{1,\infty}) = h_{\EC}(f_{1,\infty}),%
\end{eqnarray*}
as desired.%
\end{proof}

For autonomous systems, Proposition \ref{prop_entropyprops} (vi) can be used to show that the entropy depends continuously on the partition, where the set of partitions is endowed with the Rokhlin metric, given by $d_R(\PC,\QC) = H_{\mu}(\PC|\QC) + H_{\mu}(\QC|\PC)$. The nonautonomous analogue of this result is formulated in the next proposition.%

\begin{prop}\label{prop_rokhlinineq}
For two sequences $\PC_{1,\infty},\QC_{1,\infty}\in\EC_{\max}$ let%
\begin{equation*}
  d_R(\PC_{1,\infty},\QC_{1,\infty}) := \sup_{n\geq1}H_{\mu_n}(\PC_n|\QC_n) + \sup_{n\geq1}H_{\mu_n}(\QC_n|\PC_n).%
\end{equation*}
Then $d_R$ is a metric on $\EC_{\max}$ and the function $\PC_{1,\infty} \mapsto h(f_{1,\infty};\PC_{1,\infty})$ is Lipschitz continuous with Lipschitz constant $1$ on $(\EC_{\max},d_R)$.%
\end{prop}

\begin{proof}
The proof that $d_R$ is a metric easily follows from the properties of conditional entropy stated in Proposition \ref{prop_conditionalentropy}. From Proposition \ref{prop_entropyprops} (vi) we conclude the nonautonomous Rokhlin inequality%
\begin{eqnarray*}
 &&  \left|h(f_{1,\infty};\PC_{1,\infty}) - h(f_{1,\infty};\QC_{1,\infty})\right|\\
  && \leq \max\left\{\limsup_{n\rightarrow\infty}\frac{1}{n}\sum_{i=1}^n H_{\mu_i}\left(\PC_i|\QC_i\right),\limsup_{n\rightarrow\infty}\frac{1}{n}\sum_{i=1}^n H_{\mu_i}\left(\QC_i|\PC_i\right)\right\}\\
  && \leq \limsup_{n\rightarrow\infty}\frac{1}{n}\sum_{i=1}^n H_{\mu_i}\left(\PC_i|\QC_i\right) + \limsup_{n\rightarrow\infty}\frac{1}{n}\sum_{i=1}^n H_{\mu_i}\left(\QC_i|\PC_i\right)\\
  && \leq \sup_{n\geq1}H_{\mu_n}\left(\PC_n|\QC_n\right) + \sup_{n\geq1}H_{\mu_n}\left(\QC_n|\PC_n\right),%
\end{eqnarray*}
which finishes the proof.%
\end{proof}

Given a metric NDS $(X_{1,\infty},f_{1,\infty},\mu_{1,\infty})$, assume that we can decompose each of the spaces $X_n$ as a disjoint union $X_n = Y_n \dot{\cup} Z_n$ such that $f_n(Y_n) \subset Y_{n+1}$, $f_n(Z_n) \subset Z_{n+1}$, and $\mu_n(Y_n) \equiv c$ for a constant $0 < c \leq 1$. Then let us consider the restrictions of $f_{1,\infty}$ to the sequences $Y_{1,\infty} := \{Y_n\}$ and $Z_{1,\infty} := \{Z_n\}$, resp., i.e., the systems defined by the maps%
\begin{equation*}
  g_n := f_n|_{Y_n}:Y_n \rightarrow Y_{n+1},\qquad h_n := f_n|_{Z_n}:Z_n\rightarrow Z_{n+1}.%
\end{equation*}
It we consider the probability measure $\nu_n(A) := \mu_n(A)/c$ on $Y_n$, it follows that $(Y_{1,\infty},g_{1,\infty},\nu_{1,\infty})$ is also a metric system. If $c<1$, we can define a corresponding invariant sequence of probability measures for the system $(Z_{1,\infty},h_{1,\infty})$ as well.%

\begin{prop}
Let $\EC$ be an admissible class for $(X_{1,\infty},f_{1,\infty})$ and assume that $\PC_{1,\infty}\in\EC$ implies $\{\PC_n \vee \{Y_n,Z_n\}\}\in\EC$. Then%
\begin{equation*}
  \EC|_{Y_{1,\infty}} := \left\{\QC_{1,\infty}\ |\ \exists \PC_{1,\infty}\in\EC\ :\ \QC_n \equiv \{Y_n\} \vee \PC_n \right\}%
\end{equation*}
is an admissible class for $(Y_{1,\infty},g_{1,\infty})$ and%
\begin{equation*}
  c h_{\EC|_{Y_{1,\infty}}}(g_{1,\infty}) \leq h_{\EC}(f_{1,\infty}).%
\end{equation*}
If $c=1$, then equality holds.%
\end{prop}

\begin{proof}
It is clear that $\EC|_{Y_{1,\infty}}$ satisfies Axiom (A). Let $\QC_{1,\infty}\in\EC_{1,\infty}|_{Y_{1,\infty}}$. Then there exists $\PC_{1,\infty}\in\EC$ such that the elements of each $\QC_n$ are the intersections of the elements of $\PC_n$ with $Y_n$. Now assume that $\RC_{1,\infty}$ is a sequence of partitions for $Y_{1,\infty}$ which is coarser than $\QC_{1,\infty}$. Then the elements of each $\RC_n$ are unions of elements of $\QC_n$. Taking corresponding unions of elements of $\PC_n$ for each $n$, one constructs a sequence $\SC_{1,\infty}\in\EC$ coarser than $\PC_{1,\infty}$ such that $\{Y_n\} \vee \SC_{1,\infty} = \RC_{1,\infty}$, which proves that $\EC|_{Y_{1,\infty}}$ satisfies Axiom (B). Finally, if $\QC_n \equiv \{Y_n\} \vee \PC_n$ for some $\PC_{1,\infty}\in\EC$, then for all $k,m\geq1$ it holds that%
\begin{equation*}
  \bigvee_{i=0}^{m-1}g_k^{-i}\QC_{i+k} = \bigvee_{i=0}^{m-1}f_k^{-i}(\{Y_{i+k}\}\vee\PC_{i+k}) = \{Y_k\} \vee \bigvee_{i=0}^{m-1}f_k^{-i}\PC_{i+k},%
\end{equation*}
which implies that $\EC|_{Y_{1,\infty}}$ satisfies Axiom (C). To prove the inequality of entropies, consider $\QC_{1,\infty}\in\EC|_{Y_{1,\infty}}$ and the corresponding $\PC_{1,\infty}\in\EC$ with $\QC_n \equiv \{Y_n\}\vee\PC_n$. Then%
\begin{eqnarray*}
  &&  H_{\nu_1}\left(\bigvee_{i=0}^{n-1}g_1^{-i}\QC_{i+1}\right) = H_{\nu_1}\left(\{Y_1\}\vee\bigvee_{i=0}^{n-1}f_1^{-i}\PC_{i+1}\right)\\
 &=& -\frac{1}{c}\sum_{P \in \bigvee_i f_1^{-i}\PC_{i+1}}\mu_1(P \cap Y_1) \log\frac{\mu_1(P\cap Y_1)}{c}\\
  &=& -\frac{1}{c}\left[\sum_{P \in \bigvee_i f_1^{-i}\PC_{i+1}}\mu_1(P \cap Y_1) \log\mu_1(P\cap Y_1) - \sum_{P \in \bigvee_i f_1^{-i}\PC_{i+1}}\mu_1(P \cap Y_1) \log c\right].%
\end{eqnarray*}
The last summand gives%
\begin{equation*}
  \sum_{P \in \bigvee_i f_1^{-i}\PC_{i+1}}\mu_1(P \cap Y_1) \log c = \mu_1(Y_1)\log c = c \log c,%
\end{equation*}
and thus can be omitted in the computation of $h(g_{1,\infty};\QC_{1,\infty})$. We obtain%
\begin{equation*}
  h(g_{1,\infty};\QC_{1,\infty}) = \limsup_{n\rightarrow\infty}\frac{1}{n} \left[ -\frac{1}{c}\sum_{P \in \bigvee_i f_1^{-i}\PC_{i+1}}\mu_1(P \cap Y_1) \log\mu_1(P\cap Y_1)\right].%
\end{equation*}
If we consider the sequence $\widetilde{\PC}_{1,\infty}$ of partitions $\widetilde{\PC}_n := \{ P \cap Y_n : P \in \PC_n \} \cup \{ P \cap Z_n : P \in \PC_n\}$, we see that%
\begin{equation}\label{eq_entineq}
  h(g_{1,\infty};\QC_{1,\infty}) \leq \frac{1}{c}h(f_{1,\infty};\widetilde{\PC}_{1,\infty}).%
\end{equation}
By the assumption on $\EC$ it follows that $\widetilde{\PC}_{1,\infty} \in \EC$ and hence the assertion follows. In the case $c=1$, the measures $\mu_n(Z_n)$ are all zero, and hence equality holds in \eqref{eq_entineq}. Since $\widetilde{\PC}_{1,\infty}$ is finer than $\PC_{1,\infty}$, we have%
\begin{equation*}
  h_{\EC}(f_{1,\infty}) = \sup_{\widetilde{\PC}_{1,\infty}}h(f_{1,\infty};\widetilde{\PC}_{1,\infty}) = c \sup_{\QC_{1,\infty} \in \EC|_{Y_{1,\infty}}}h(g_{1,\infty};\QC_{1,\infty}) = ch_{\EC|_{Y_{1,\infty}}}(g_{1,\infty}),%
\end{equation*}
which finishes the proof.%
\end{proof}

\begin{remark}
For a topological NDS given by a sequence of homeomorphisms, endowed with an invariant sequence of Borel probability measures, the above proposition can be applied to the decomposition $Y_n := \supp\mu_n$, $Z_n := X_n \backslash \supp\mu_n$, where $\supp\mu_n = \{x \in X_n | \forall \eps>0: \mu_n(B(x,\eps))>0\}$ is the support of the measure $\mu_n$.%
\end{remark}

\subsection{The Power Rule for Metric Entropy}%

Given a metric NDS $(X_{1,\infty},f_{1,\infty})$ and $k\in\N$, we define the $k$-th power system $(X^{[k]}_{1,\infty},f^{[k]}_{1,\infty})$ in exactly the same way as we did for topological systems. It is very easy to see that this system is a metric system as well.%

If $\EC$ is an admissible class for $(X_{1,\infty},f_{1,\infty})$, we denote by $\EC^{[k]}$ the class of all sequences of partitions for $X^{[k]}_{1,\infty}$ which are defined by restricting the sequences in $\EC$ to the spaces in $X^{[k]}_{1,\infty}$, i.e., $\PC_{1,\infty} = \{\PC_n\}\in\EC$ iff%
\begin{equation*}
  \PC^{[k]}_{1,\infty} := \{\PC_{(n-1)k+1}\}_{n\geq1} \in \EC^{[k]}.%
\end{equation*}

\begin{prop}\label{prop_metricpowerrule}
If $\EC$ is an admissible class for $(X_{1,\infty},f_{1,\infty})$, then $\EC^{[k]}$ is an admissible class for $(X^{[k]}_{1,\infty},f^{[k]}_{1,\infty})$ and%
\begin{equation*}
  h_{\EC^{[k]}}\left(f^{[k]}_{1,\infty}\right) = k\cdot h_{\EC}\left(f_{1,\infty}\right).%
\end{equation*}
\end{prop}

\begin{proof}
It is clear that $\EC^{[k]}$ satisfies Axiom (A). To verify Axiom (B), consider $\PC_{1,\infty}^{[k]}\in\EC^{[k]}$ for some $\PC_{1,\infty}\in\EC$. If $\QC_{1,\infty}$ is a sequence of partitions for $X^{[k]}_{1,\infty}$ which is coarser than $\PC^{[k]}_{1,\infty}$ (i.e., $\QC_n \preceq \PC_{(n-1)k+1}$ for all $n\geq1$), we can extend $\QC_{1,\infty}$ to a sequence $\RC_{1,\infty}$ of partitions for $X_{1,\infty}$ which is coarser than $\PC_{1,\infty}$. This can be done in a trivial way by putting%
\begin{equation*}
  \RC_n := \left\{\begin{array}{ll}
                    \PC_n & \mbox{if } n-1 \mbox{ is not a multiple of } k,\\
                    \QC_{1+(n-1)/k} & \mbox{if } n-1 \mbox{ is a multiple of } k.%
                  \end{array}\right.%
\end{equation*}
Then it follows that $\RC_n = \PC_n \preceq \PC_n$ in the first case, and $\RC_n = \QC_{1+(n-1)/k} \preceq \PC_n$ in the second one. Since $\EC$ satisfies Axiom (B), we know that $\RC_{1,\infty}\in\EC$, which implies that
$\QC_{1,\infty} = \RC^{[k]}_{1,\infty} \in \EC^{[k]}$. To show that $\EC^{[k]}$ satisfies Axiom (C), let $\PC_{1,\infty}\in\EC$ and $m\geq1$. We have to show that the sequence $\QC_{1,\infty}$ defined by%
\begin{equation*}
  \QC_n := \bigvee_{i=0}^{m-1}f_{(n-1)k+1}^{-ik}\PC_{(i+n-1)k+1}%
\end{equation*}
is an element of $\EC^{[k]}$. To this end, first note that%
\begin{eqnarray*}
  \QC_n \preceq \bigvee_{i=0}^{mk-1}f_{(n-1)k+1}^{-i}\PC_{(n-1)k+1+i} =: \RC_n.%
\end{eqnarray*}
The sequence $\RC_{1,\infty}$ can be extended to an element $\SC_{1,\infty}$ of $\EC$, which is given by%
\begin{equation*}
  \SC_n := \bigvee_{i=0}^{mk-1}f_n^{-i}\PC_{n+i}.%
\end{equation*}
Indeed, $\SC_{1,\infty}\in\EC$, since $\EC$ satisfies Axiom (C). Hence, $\RC_{1,\infty} = \SC_{1,\infty}^{[k]} \in \EC^{[k]}$ and since $\EC^{[k]}$ satisfies Axiom (B), this implies $\QC_{1,\infty}\in\EC^{[k]}$. Now let us prove the formula for the entropies. Let $\PC_{1,\infty}\in\EC$. We define a sequence $\QC_{1,\infty}$ of finite measurable partitions for $X^{[k]}_{1,\infty}$ as follows:%
\begin{equation*}
  \QC_n := \bigvee_{j=0}^{k-1}f^{-j}_{(n-1)k+1}\PC_{(n-1)k+1+j}.%
\end{equation*}
The sequence $\QC_{1,\infty}$ is an element of $\EC^{[k]}$, since it is of the form $\QC_{1,\infty} = \RC^{[k]}_{1,\infty}$ with $\RC_{1,\infty}\in\EC$. This follows by combining the facts that $\PC_{1,\infty}\in\EC$ and $\EC$ satisfies Axiom (C). We find that%
\begin{eqnarray*}
  h\left(f^{[k]}_{1,\infty};\QC_{1,\infty}\right) &=& \limsup_{n\rightarrow\infty}\frac{1}{n}H_{\mu_1}\left(\bigvee_{i=0}^{n-1}f_1^{-ik}\QC_{i+1}\right)\allowdisplaybreaks\\
                                                  &=& \limsup_{n\rightarrow\infty}\frac{1}{n}H_{\mu_1}\left(\bigvee_{i=0}^{n-1}f_1^{-ik}\bigvee_{j=0}^{k-1}f^{-j}_{ik+1}\PC_{ik+1+j}\right)\allowdisplaybreaks\\
                                                  &=& \limsup_{n\rightarrow\infty}\frac{1}{n}H_{\mu_1}\left(\bigvee_{i=0}^{n-1}\bigvee_{j=0}^{k-1}f_1^{-(ik+j)}\PC_{(ik+j)+1}\right)\allowdisplaybreaks\\
                                                  &=& k\cdot\limsup_{n\rightarrow\infty}\frac{1}{nk}H_{\mu_1}\left(\bigvee_{i=0}^{nk-1}f_1^{-i}\PC_{i+1}\right) = k\cdot h\left(f_{1,\infty};\PC_{1,\infty}\right).%
\end{eqnarray*}
To obtain the last equality we used Proposition \ref{prop_entropyprops} (iv). Now consider also the sequence $\PC_{1,\infty}^{[k]}$. It is obvious that $\QC_{1,\infty}$ is finer than $\PC^{[k]}_{1,\infty}$. Hence, using Proposition \ref{prop_entropyprops} (iii), we find%
\begin{equation*}
  h\left(f^{[k]}_{1,\infty};\PC_{1,\infty}^{[k]}\right) \leq h\left(f^{[k]}_{1,\infty};\QC_{1,\infty}\right) = k \cdot h\left(f_{1,\infty};\PC_{1,\infty}\right).%
\end{equation*}
Taking the supremum over all $\PC_{1,\infty}^{[k]}$ on the left-hand side and over all $\PC_{1,\infty}$ on the right-hand side, the inequality%
\begin{equation*}
  h_{\EC^{[k]}}\left(f^{[k]}_{1,\infty}\right) \leq k \cdot h_{\EC}\left(f^{[k]}_{1,\infty}\right)%
\end{equation*}
follows. The converse inequality follows from%
\begin{equation*}
  h_{\EC^{[k]}}\left(f^{[k]}_{1,\infty}\right) \geq h\left(f^{[k]}_{1,\infty};\QC_{1,\infty}\right) = k \cdot h\left(f_{1,\infty};\PC_{1,\infty}\right),%
\end{equation*}
which holds for every $\PC_{1,\infty}\in\EC$.%
\end{proof}

\section{Relation to Topological Entropy}\label{sec_vi}%

In order to prove a variational inequality, we consider a topological NDS $(X_{1,\infty},f_{1,\infty})$ with an $f_{1,\infty}$-invariant sequence $\mu_{1,\infty}$ of Borel probability measures. When speaking of measurable partitions in this context, we mean ``exact'' partitions and not partitions in the sense of measure theory, where different elements of the partition may have a nonempty overlap of measure zero. We will frequently use the property of inner regularity of Borel measures, i.e., $\mu(A) = \sup\{\mu(K) : K \subset A \mbox{ compact}\}$ for any Borel subset of a compact metric space.%

\subsection{The Misiurewicz Class}%

In this subsection, we introduce a special admissible class which we will use to prove the variational inequality. This class is constructed in such a way that its elements are just perfect to apply the arguments of Misiurewicz's proof of the variational principle to them. Therefore, we call it the \emph{Misiurewicz class}.%

Let $(X_{1,\infty},f_{1,\infty})$ be a topological NDS with an $f_{1,\infty}$-invariant sequence of Borel probability measures $\mu_{1,\infty} = \{\mu_n\}$.%

We define the Misiurewicz class $\EC_{\Mis} \subset \EC_{\max}$ as follows. A sequence $\PC_{1,\infty}\in\EC_{\max}$, $\PC_n = \{P_{n,1},\ldots,P_{n,k_n}\}$, is an element of $\EC_{\Mis}$ iff for every $\eps>0$ there exist $\delta>0$ and compact sets $C_{n,i}\subset P_{n,i}$ ($n\geq1$, $1 \leq i \leq k_n$) such that for every $n\geq1$ the following two hypotheses are satisfied:%
\begin{enumerate}
\item[(a)] $\mu_n(P_{n,i}\backslash C_{n,i}) \leq \eps$.%
\item[(b)] The minimal distance between the sets $C_{n,i}$ is at least $\delta$, i.e.,%
\begin{equation*}
  \min_{1\leq i<j \leq k_n} \min\left\{\varrho_n(x,y)\ :\ (x,y) \in C_{n,i} \tm C_{n,j}\right\} \geq \delta.%
\end{equation*}
\end{enumerate}

\begin{prop}\label{prop_emis_admissibleclass}
If $f_{1,\infty}$ is equicontinuous, then $\EC_{\Mis}$ is an admissible class.%
\end{prop}

\begin{proof}
First note that $\EC_{\Mis}$ is nonempty, since it contains the trivial sequence defined by $\PC_n := \{X_n\}$ for all $n\geq1$. To show that $\EC_{\Mis}$ satisfies Axiom (B), assume that $\PC_{1,\infty} = \{\PC_n\}\in\EC_{\Mis}$, $\PC_n = \{P_{n,1},\ldots,P_{n,k_n}\}$, and let $\QC_{1,\infty}$ be a sequence which is coarser than $\PC_{1,\infty}$. Let $\QC_n$ be given by%
\begin{equation*}
  \QC_n = \{Q_{n,1},\ldots,Q_{n,l_n}\}.%
\end{equation*}
Then every element of $\QC_n$ must be a disjoint union of elements of $\PC_n$:%
\begin{equation*}
  Q_{n,i} = \bigcup_{\alpha=1}^{N_{n,i}}P_{n,j_{\alpha}}.%
\end{equation*}
Since $\PC_{1,\infty}\in\EC_{\Mis}$, we can choose compact sets $C_{n,i}\subset P_{n,i}$ and $\delta>0$ depending on a given $\eps = \widetilde{\eps}/(\max_{n\geq1}\#\PC_n)$ such that (a) and (b) hold for $\PC_{1,\infty}$. Define%
\begin{equation*}
  D_{n,i} := \bigcup_{\alpha=1}^{N_{n,i}}C_{n,j_{\alpha}},\qquad n\geq1,\ i=1,\ldots,l_n.%
\end{equation*}
It is clear that $D_{n,i}$ is a compact subset of $Q_{n,i}$. Moreover, it holds that%
\begin{eqnarray*}
  && \mu_n\left(Q_{n,i} \backslash D_{n,i}\right) = \mu_n\left(\bigcup_{\alpha=1}^{N_{n,i}}P_{n,j_{\alpha}} \backslash \bigcup_{\alpha=1}^{N_{n,i}}C_{n,j_{\alpha}}\right)\\
     && = \mu_n\left(\bigcup_{\alpha=1}^{N_{n,i}}[P_{n,j_{\alpha}}\backslash C_{n,j_{\alpha}}]\right) = \sum_{\alpha=1}^{N_{n,i}}\mu_n\left(P_{n,j_{\alpha}}\backslash C_{n,j_{\alpha}}\right)
      \leq \frac{N_{n,i}\widetilde{\eps}}{\max_{n\geq1}\#\PC_n} \leq \widetilde{\eps}.%
\end{eqnarray*}
For $i\neq j$ we have%
\begin{eqnarray*}
  && \min\left\{\varrho_n(x,y)\ :\ (x,y) \in \bigcup_{\alpha=1}^{N_{n,i}}C_{n,j_{\alpha}} \tm \bigcup_{\beta=1}^{N_{n,j}}C_{n,j_{\beta}}\right\}\\
  && = \min_{\alpha,\beta} \min\left\{\varrho_n(x,y)\ :\ (x,y) \in C_{n,j_{\alpha}} \tm C_{n,j_{\beta}}\right\} \geq \delta,%
\end{eqnarray*}
since each $C_{n,j_{\alpha}}$ is disjoint from all $C_{n,j_{\beta}}$. Hence, $\QC_{1,\infty}\in\EC_{\Mis}$. To show that Axiom (C) holds, let $\PC_{1,\infty}=\{\PC_n\}\in\EC_{\Mis}$, $\PC_n = \{P_{n,1},\ldots,P_{n,k_n}\}$, and $m\geq1$. Consider the sequence $\PC^{\langle m\rangle}_{1,\infty}(f_{1,\infty})$. For given $\eps = (1/m)\widetilde{\eps}>0$ choose $\delta>0$ and compact sets $C_{n,i}\subset P_{n,i}$ such that (a) and (b) hold for $\PC_{1,\infty}$. For every $r\geq1$ and $(j_0,\ldots,j_{m-1})\in \{1,\ldots,k_r\} \tm \cdots \tm \{1,\ldots,k_{r+m-1}\}$ define%
\begin{equation*}
  D_{r,(j_0,\ldots,j_{m-1})} := \bigcap_{i=0}^{m-1}f_r^{-i}(C_{r+i,j_i}).%
\end{equation*}
These sets are obviously compact subsets of $X_r$ and each element of $\PC^{\langle m \rangle}_r(f_{1,\infty})$ contains exactly one such set. We have%
\begin{eqnarray*}
  && \mu_r\left(\bigcap_{i=0}^{m-1}f_r^{-i}(P_{r+i,j_i}) \backslash \bigcap_{i=0}^{m-1}f_r^{-i}(C_{r+i,j_i})\right)\allowdisplaybreaks\\
   && = \mu_r\left(\bigcup_{l=0}^{m-1}\left[\left(\bigcap_{i=0}^{m-1}f_r^{-i}(P_{r+i,j_i})\right)\backslash f_r^{-l}(C_{r+l,j_l})\right]\right)\allowdisplaybreaks\\
   && \leq \sum_{l=0}^{m-1}\mu_r\left(f_r^{-l}(P_{r+l,j_l}) \backslash f_r^{-l}(C_{r+l,j_l})\right)\allowdisplaybreaks\\
   && = \sum_{l=0}^{m-1} f_r^l\mu_r\left(P_{r+l,j_l} \backslash C_{r+l,j_l}\right) = \sum_{l=0}^{m-1} \mu_{r+l}\left(P_{r+l,j_l} \backslash C_{r+l,j_l}\right) \leq m\eps = \widetilde{\eps}.%
\end{eqnarray*}
Finally, in order to show that (b) holds for $\PC^{\langle m\rangle}(f_{1,\infty})$, we need the assumption of equicontinuity for $f_{1,\infty}$, which yields a number $\rho>0$ such that $\varrho_r(x,y)<\rho$ implies $\varrho_{r+i}(f_r^i(x),f_r^i(y)) < \delta$ for all $r\geq1$ and $i=0,1,\ldots,m-1$ (cf.~the proof of Lemma \ref{lem_topologicalrefinements}). Now consider two sets $D_{r,(j_0,\ldots,j_{m-1})}$ and $D_{r,(l_0,\ldots,l_{m-1})}$. These sets are disjoint iff there is an index $\alpha \in \{0,1,\ldots,m-1\}$ such that $j_{\alpha} \neq l_{\alpha}$. This implies $\varrho_{r+\alpha}(f_r^{\alpha}(x),f_r^{\alpha}(y)) \geq \delta$, and hence $\varrho_r(x,y) \geq \rho$. Thus, we have found that for every $r\geq1$ it holds that%
\begin{equation*}
  \min_{(j_0,\ldots,j_{m-1}) \neq \atop \quad(l_0,\ldots,l_{m-1})} \min\left\{\varrho_r(x,y)\ :\ (x,y) \in D_{r,(j_0,\ldots,j_{m-1})} \tm D_{r,(l_0,\ldots,l_{m-1})}\right\} \geq \rho,%
\end{equation*}
which completes the proof.%
\end{proof}

In \cite[Thm.~B]{KSn} it is shown that an equiconjugacy preserves the topological entropy of a topological NDS. An equiconjugacy between systems $f_{1,\infty}$ and $g_{1,\infty}$ is an equicontinuous sequence $\pi_{1,\infty} = \{\pi_n\}$ of homeomorphisms such that also $\{\pi_n^{-1}\}$ is equicontinuous and $\pi_{n+1}\circ f_n = g_n\circ \pi_n$. The following proposition shows that an equiconjugacy also preserves the Misiurewicz class and hence the associated metric entropy.%

\begin{prop}\label{prop_emisequiconj}
Consider two equicontinuous topological NDSs $(X_{1,\infty},f_{1,\infty})$ and $(Y_{1,\infty},g_{1,\infty})$. Assume that $\pi_{1,\infty}$ is an equisemiconjugacy from $f_{1,\infty}$ to $g_{1,\infty}$, i.e., it holds that $\pi_{n+1}\circ f_n = g_n \circ \pi_n$ for all $n\geq1$ and the sequence $\{\pi_n\}$ is equicontinuous. Then, if $\mu_{1,\infty}$ is an $f_{1,\infty}$-invariant sequence, $\nu_{1,\infty} = \{\nu_n\}$, $\nu_n := \pi_n\mu_n$, is $g_{1,\infty}$-invariant and $\pi_{1,\infty}$ is an $\EC_{\Mis}(f_{1,\infty})$-$\EC_{\Mis}(g_{1,\infty})$-semiconjugacy. Hence,%
\begin{equation*}
  h_{\EC_{\Mis}}(g_{1,\infty}) \leq h_{\EC_{\Mis}}(f_{1,\infty}).%
\end{equation*}
\end{prop}

\begin{proof}
We have $g_n\nu_n = g_n(\pi_n\mu_n) = \pi_{n+1}f_n\mu_n = \pi_{n+1}\mu_{n+1} = \nu_{n+1}$ and hence, $\nu_{1,\infty}$ is $g_{1,\infty}$-invariant. To show that $\pi_{1,\infty}$ is an $\EC_{\Mis}(f_{1,\infty})$-$\EC_{\Mis}(g_{1,\infty})$-semiconjugacy, consider some $\QC_{1,\infty} \in \EC_{\Mis}(g_{1,\infty})$ and let $\PC_n := \{\pi_n^{-1}(Q) : Q \in \QC_n\}$ for all $n\geq1$. To show that $\PC_{1,\infty}\in\EC_{\Mis}(f_{1,\infty})$, let $\eps>0$. Then, if $\QC_n = \{Q_{n,1},\ldots,Q_{n,k_n}\}$, we find compact sets $C_{n,i}\subset Q_{n,i}$ and $\delta>0$ such that $\nu_n(Q_{n,i}\backslash C_{n,i})\leq\eps$ and%
\begin{equation}\label{eq_min}
  \min_{1\leq i<j \leq k_n}\min\left\{\varrho_{Y_n}(y_1,y_2)\ :\ (y_1,y_2) \in C_{n,i}\tm C_{n,j}\right\} \geq \delta.%
\end{equation}
Since $\{\pi_n\}$ is equicontinuous, there exists $\rho>0$ such that $\varrho_{X_n}(x_1,x_2) < \rho$ implies $\varrho_{Y_n}(\pi_n(x_1),\pi_n(x_2)) < \delta$ for all $n\geq1$ and $x_1,x_2\in X_n$. Now consider the closed (and hence compact) sets $\pi_n^{-1}(C_{n,i}) \subset \pi_n^{-1}(Q_{n,i}) =: P_{n,i} \in \PC_n$. We have $\mu_n(P_{n,i} \backslash \pi_n^{-1}(C_{n,i})) = \nu_n(Q_{n,i}\backslash C_{n,i}) \leq \eps$. Assume to the contrary that there exist $n\in\N$, $i\neq j$, and $x_1 \in \pi_n^{-1}(C_{n,i})$, $x_2\in\pi_n^{-1}(C_{n,j})$, such that $\varrho_{X_n}(x_1,x_2) < \rho$. This implies $\varrho_{Y_n}(\pi_n(x_1),\pi_n(x_2)) < \delta$. Since $\pi_n(x_1) \in C_{n,i}$ and $\pi_n(x_2)\in C_{n,j}$, this contradicts \eqref{eq_min}. Hence, $\PC_{1,\infty}\in\EC_{\Mis}(f_{1,\infty})$ and the rest follows from Proposition \ref{prop_conjugacy}.%
\end{proof}

\subsection{The Variational Inequality}%

Now we are in position to prove the general variational inequality following the lines of Misiurewicz's proof \cite{Mis}.%

\begin{theo}\label{thm_general_varinequality}
For an equicontinuous topological NDS $(X_{1,\infty},f_{1,\infty})$ with an invariant sequence $\mu_{1,\infty}$ it holds that%
\begin{equation*}
  h_{\EC_{\Mis}}(f_{1,\infty}) \leq h_{\tp}(f_{1,\infty}).%
\end{equation*}
\end{theo}

\begin{proof}
Let $\PC_{1,\infty}\in\EC_{\Mis}$. We may assume that each $\PC_n$ has the same number $k$ of elements, $\PC_n = \{P_{n,1},\ldots,P_{n,k}\}$. By definition of the Misiurewicz class, we find compact sets $Q_{n,i}\subset P_{n,i}$ (for all $n,i$) such that%
\begin{equation*}
  \mu_n(P_{n,i}\backslash Q_{n,i}) \leq \frac{1}{k\log k},\qquad i=1,\ldots,k,\ n\geq1,%
\end{equation*}
and $\delta>0$ with%
\begin{equation}\label{eq_delta}
  \min_{1\leq i<j \leq k} \min\left\{\varrho_n(x,y)\ :\ (x,y) \in Q_{n,i} \tm Q_{n,j}\right\} \geq \delta.%
\end{equation}
By setting $Q_{n,0} := X_n \backslash \bigcup_{i=1}^k Q_{n,i}$ we can define another sequence $\QC_{1,\infty}$ of measurable partitions $\QC_n := \{Q_{n,0},Q_{n,1},\ldots,Q_{n,k}\}$. As in Misiurewicz's proof one finds $H_{\mu_n}\left(\PC_n|\QC_n\right) \leq 1$, which by Proposition \ref{prop_entropyprops} (vi) leads to the inequality%
\begin{equation}\label{eq_partitionscomparison}
  h\left(f_{1,\infty};\PC_{1,\infty}\right) \leq h\left(f_{1,\infty};\QC_{1,\infty}\right) + 1.%
\end{equation}
Define a sequence $\UC_{1,\infty}$ of open covers $\UC_n$ of $X_n$ by%
\begin{equation*}
  \UC_n := \left\{Q_{n,0} \cup Q_{n,1},\ldots,Q_{n,0} \cup Q_{n,k}\right\}.%
\end{equation*}
To see that the sets $Q_{n,0}\cup Q_{n,i}$ are open, consider their complements $Q_{n,1} \cup \ldots \cup Q_{n,i-1} \cup Q_{n,i+1} \cup \ldots \cup Q_{n,k}$, which are finite unions of compact sets and hence closed. For a fixed $m\geq1$, let $E_m \subset X_1$ be a maximal $(m,\delta)$-separated set. From \eqref{eq_delta} it follows that each $(\delta/2)$-ball in $X_n$ intersects at most two elements of $\QC_n$ for any $n\geq1$.
Hence, we can associate to each $x\in E_m$ at most $2^m$ different elements of $\bigvee_{i=0}^{m-1}f_1^{-i}\QC_{i+1}$, which implies%
\begin{equation*}
  \#\left[\bigvee_{i=0}^{m-1}f_1^{-i}\QC_{i+1}\right] \leq 2^m r_{\sep}\left(m,\frac{\delta}{2},f_{1,\infty}\right).%
\end{equation*}
Consequently, we obtain%
\begin{equation*}
  H_{\mu_1}\left(\bigvee_{i=0}^{m-1}f_1^{-i}\QC_{i+1}\right) \leq \log\#\left[\bigvee_{i=0}^{m-1}f_1^{-i}\QC_{i+1}\right] \leq \log r_{\sep}\left(m,\frac{\delta}{2},f_{1,\infty}\right) + m \log 2.%
\end{equation*}
Using \eqref{eq_partitionscomparison}, we therefore have%
\begin{eqnarray*}
  h\left(f_{1,\infty};\PC_{1,\infty}\right) &\leq& \limsup_{m\rightarrow\infty}\frac{1}{m}\log r_{\sep}\left(m,\frac{\delta}{2},f_{1,\infty}\right) + \log 2 + 1\\
  &\leq& h_{\tp}(f_{1,\infty}) + \log 2 + 1.%
\end{eqnarray*}
Taking the supremum over all $\PC_{1,\infty}\in\EC_{\Mis}$, we find%
\begin{equation*}
  h_{\EC_{\Mis}}(f_{1,\infty}) \leq h_{\tp}(f_{1,\infty}) + \log 2 + 1.%
\end{equation*}
That the constant term $\log 2 + 1$ can be omitted in this estimate now follows from a careful application of the power rules for topological and metric entropy. Inspecting the definition of the Misiurewicz class, one sees that for every $k\geq1$ the admissible class $\EC_{\Mis}^{[k]}$ is contained in the Misiurewicz class of $f^{[k]}_{1,\infty}$. Therefore, the arguments that we have applied to the system $(X_{1,\infty},f_{1,\infty})$ can equally be applied to all of the power systems $(X^{[k]}_{1,\infty},f^{[k]}_{1,\infty})$, $k\geq1$. Hence, using the power rules (Proposition \ref{prop_powerrule} and Proposition \ref{prop_metricpowerrule}), we obtain%
\begin{equation*}
  h_{\EC_{\Mis}}(f_{1,\infty}) \leq h_{\tp}(f_{1,\infty}) + \frac{\log 2 + 1}{k}.%
\end{equation*}
Since this holds for every $k\geq1$, sending $k$ to infinity gives the result.%
\end{proof}

An interesting corollary of Theorem \ref{thm_general_varinequality} is the following generalized variational principle for autonomous systems.%

\begin{cor}
For a topological autonomous system $(X,f)$ it holds that%
\begin{equation*}
  \sup_{\mu_{1,\infty}}h_{\EC_{\Mis}(f,\mu_{1,\infty})}(f) = h_{\tp}(f),%
\end{equation*}
where the supremum is taken over all sequences $\mu_{1,\infty}$ with $f\mu_n \equiv \mu_{n+1}$.%
\end{cor}

\begin{proof}
The inequality ``$\leq$'' holds by Theorem \ref{thm_general_varinequality}. The converse inequality follows from the classical variational principle, if we consider only the constant sequences $\mu_{1,\infty}$, i.e., the invariant measures of $f$, and assure ourselves that the associated Misiurewicz classes contain all constant sequences.%
\end{proof}

\begin{cor}
Let $f_{1,\infty}$ be an equicontinuous sequence of (not necessarily strictly) monotone maps $f_n:X\rightarrow X$ where $X$ is either a compact interval or a circle. Then for every $f_{1,\infty}$-invariant sequence $\mu_{1,\infty}$ it holds that $h_{\EC_{\Mis}}(f_{1,\infty})=0$.%
\end{cor}

\begin{proof}
This follows from \cite[Thm.~D]{KSn}, which asserts that the corresponding topological entropy is zero.%
\end{proof}

\subsection{Large Misiurewicz Classes} 

Up to now, we only know that the Misiurewicz class $\EC_{\Mis}$ contains the trivial sequence of partitions. If it would contain no further sequences, Theorem \ref{thm_general_varinequality} would not give any valuable information on the metric or topological entropy. The aim of this subsection is to find conditions on invariant sequences of measures which give rise to a large Misiurewicz class. The simplest case consists in a system $(X_{1,\infty},f_{1,\infty},\mu_{1,\infty})$, where both $X_{1,\infty}$ and $\mu_{1,\infty}$ are constant, say $X_n\equiv X$ and $\mu_n\equiv\mu$. Then any finite measurable partition $\PC$ of $X$ gives rise to a constant sequence $\PC_n \equiv \PC$ of partitions which is obviously contained in $\EC_{\Mis}$. The following proposition slightly generalizes this situation.%

\begin{prop}
Let $(X_{1,\infty},f_{1,\infty})$ be an equicontinuous NDS with an $f_{1,\infty}$-invariant sequence $\mu_{1,\infty}$. If $X_{1,\infty}$ is constant and the closure of $\{\mu_n\}$ with respect to the strong topology on the space of probability measures is compact, then $\EC_{\Mis}$ contains all constant sequences of partitions.%
\end{prop}

\begin{proof}
We first show that every Borel set $A\subset X$ can be approximated by compact subsets uniformly for all $\mu_n$. The strong topology is characterized by%
\begin{equation*}
  \mu_n \rightarrow \mu \quad \Leftrightarrow \quad \mu_n(A) \rightarrow \mu(A) \mbox{ for every Borel set } A \subset X.%
\end{equation*}
Let $\CC$ be the strong closure of $\mu_{1,\infty}$, and let $A\subset X$ be a Borel set and $\eps>0$. For each $\mu\in\CC$ there exists a compact set $B_{\mu}\subset A$ such that $\mu(A\backslash B_{\mu})\leq \eps/2$. Now take a neighborhood $\UC_{\mu}$ of $\mu$ in $\CC$ such that $|\nu(A\backslash B_{\mu}) - \mu(A\backslash B_{\mu})|\leq\eps/2$ for all $\nu\in\UC_{\mu}$. Then for every $\nu\in\UC_{\mu}$ we have%
\begin{equation*} 
  \nu(A\backslash B_{\mu}) \leq \mu(A\backslash B_{\mu}) + \frac{\eps}{2} \leq \eps.%
\end{equation*}
We can cover the compact set $\CC$ by finitely many of such neighborhoods, say $\UC_{\mu_1},\ldots,\UC_{\mu_r}$. Then $B := \bigcup_{i=1}^r B_{\mu_i}$ is a compact subset of $A$ which satisfies $\nu(A\backslash B) \leq \eps$ for all $\nu\in\CC$, so in particular for all $\nu = \mu_n$. Now let $\PC = \{P_1,\ldots,P_k\}$ be a finite measurable partition of the state space $X$. Then for any given $\eps>0$ we find compact sets $C_i\subset P_i$ such that $\mu_n(P_i \backslash C_i) \leq \eps$ for all $n\geq1$ and $i=1,\ldots,k$. Moreover, since the sets $C_i$ are pairwisely disjoint,%
\begin{equation*}
  \min_{1\leq i<j \leq k} \min\left\{\varrho(x,y)\ :\ (x,y) \in C_i \tm C_j\right\} > 0.%
\end{equation*}
This implies that the constant sequence $\PC_n\equiv\PC$ is an element of $\EC_{\Mis}$.%
\end{proof}

\begin{example}
Consider a system which is given by a periodic sequence%
\begin{equation*}
  f_{1,\infty} = \left\{f_1,f_2,\ldots,f_N,f_1,f_2,\ldots,f_N,\ldots\right\}.%
\end{equation*}
Let $\mu_1$ be an $f_1^N$-invariant probability measure on $X$ (which exists by the theorem of Krylov-Bogolyubov). Define%
\begin{equation*}
  \mu_2 := f_1\mu_1,\quad \mu_3 := f_2\mu_2,\quad\ldots,\quad\mu_N := f_{N-1}\mu_{N-1},%
\end{equation*}
and extend this to an $N$-periodic sequence%
\begin{equation*}
  \mu_{1,\infty} = \left\{\mu_1,\mu_2,\ldots,\mu_N,\mu_1,\mu_2,\ldots,\mu_N,\ldots\right\}.%
\end{equation*}
Then $\mu_{1,\infty}$ is an $f_{1,\infty}$-invariant sequence, which follows from%
\begin{equation*}
  f_N\mu_N = f_N f_{N-1}\mu_{N-1} = f_N f_{N-1} f_{N-2}\mu_{N-2} = \cdots = f_1^N \mu_1 = \mu_1.%
\end{equation*}
Clearly, $\{\mu_1,\ldots,\mu_N\}$ is compact.%
\end{example}

The assumption that the closure of $\{\mu_n\}$ should be compact still seems to be very restrictive. The next result provides another condition.%

\begin{lemma}\label{lem_basets}
Let $(X,\varrho)$ be a compact metric space with a Borel probability measure $\mu$. Let $A\subset X$ be a Borel set with $\mu(\partial A)=0$. Then $A$ can be approximated by compact subsets with zero boundaries, i.e.,%
\begin{equation*}
  \mu(A) = \sup\left\{\mu(K)\ :\ K \subset A \mbox{ compact with } \mu(\partial K) = 0\right\}.%
\end{equation*}
\end{lemma}

\begin{proof}
We can assume without loss of generality that $\partial A \neq \emptyset$, since otherwise $A$ is closed and hence compact itself. For every $\eps>0$ define the set%
\begin{equation*}
  K_{\eps} := \left\{x\in\inner A\ :\ \dist(x,\partial A) \geq \eps\right\}.%
\end{equation*}
We claim that each $K_{\eps}$ is a closed subset of $X$ and hence compact. To this end, consider a sequence $x_n \in K_{\eps}$ with $x_n \rightarrow x\in X$. By continuity of $\dist(\cdot,\partial A)$, it follows that $\dist(x,\partial A)\geq\eps$ and $x\in\cl A$. Assume to the contrary that $x\in\partial A$. Then $\eps\leq\dist(x,\partial A)=0$, a contradiction. Hence, $x\in K_{\eps}$. We further claim that $\mu(K_{\eps}) \rightarrow \mu(A)$ for $\eps\rightarrow0$. To show this, take an arbitrary strictly decreasing sequence $\eps_n\rightarrow 0$. Then $K_{\eps_n} \subset K_{\eps_{n+1}}$ for all $n\geq1$. Hence, by continuity of the measure $\mu$ and the assumption that $\mu(\partial A)=0$, it follows that%
\begin{equation*}
  \mu(A) = \mu(\inner A) = \mu\left(\bigcup_{n\geq1} K_{\eps_n}\right) = \lim_{n\rightarrow\infty}\mu(K_{\eps_n}).%
\end{equation*}
To conclude the proof, it suffices to show that one can choose the sequence $\eps_n$ such that $\mu(\partial K_{\eps_n})=0$. To this end, we first show that for $\delta_1 < \delta_2$ the boundaries of $K_{\delta_1}$ and $K_{\delta_2}$ are disjoint. Assume to the contrary that there exists $x \in \partial K_{\delta_1}\cap \partial K_{\delta_2}$. Then, by continuity of the $\dist$-function, $\dist(x,\partial A)\geq\delta_1$ and $\dist(x,\partial A)\geq\delta_2$. However, if one of these inequalities would be strict, the point $x$ would be contained in the interior of the corresponding set. Hence, $\dist(x,\partial A) = \delta_1 < \delta_2 = \dist(x,\partial A)$, a contradiction. Now, we can construct the desired sequence $\eps_n\rightarrow0$ as follows. Fix $n\in\N$ and assume to the contrary that $\mu(\partial K_{\eps}) > 0$ for all $\eps\in(1/(n+1),1/n)$. Define the sets $I_m := \{\eps\in(1/(n+1),1/n) : \mu(\partial K_{\eps})\geq 1/m\}$. Then $(1/(n+1),1/n) = \bigcup_{m\in\N}I_m$ and hence one of the sets $I_m$, say $I_{m_0}$, must be uncountable. However, since the boundaries of the $K_{\eps}$ are disjoint, this would imply that the set $\bigcup_{\eps\in I_{m_0}} \partial K_{\eps}$ has an infinite measure. Hence, we can take $\eps_n \in (1/(n+1),1/n)$ with $\mu(\partial K_{\eps_n}) = 0$.%
\end{proof}

\begin{prop}\label{prop_weaklystablesuffcon}
Let $(X_{1,\infty},f_{1,\infty})$ be an equicontinuous system such that $X_{1,\infty}$ is constant and let $\mu_{1,\infty} = \{\mu_n\}$ be an $f_{1,\infty}$-invariant sequence. Assume that the measures in the weak$^*$-closure of $\{\mu_n\}$ are pairwisely equivalent. Then $\EC_{\Mis}$ contains all constant sequences of partitions whose members have zero boundaries (with respect to the measures $\mu_n$).%
\end{prop}

\begin{proof}
Let $\CC$ denote the weak$^*$-closure of $\{\mu_n\}$. Consider a finite measurable partition $\PC = \{P_1,\ldots,P_k\}$ of the state space $X$ such that $\nu(\partial P_i)=0$, $1\leq i\leq k$, for one and hence all $\nu\in\CC$. Fix $\eps>0$ and pick $\nu\in\CC$. By Lemma \ref{lem_basets}, we find compact sets $C_{\nu,i} \subset P_i$ with $\nu(\partial C_{\nu,i}) = 0$, $1\leq i\leq k$, and%
\begin{equation*}
  \nu(P_i \backslash C_{\nu,i}) \leq \eps/2,\qquad 1\leq i\leq k.%
\end{equation*}
Since $\partial(P_i \backslash C_{\nu,i})\subset\partial P_i \cup \partial C_{\nu,i}$ and hence $\nu(\partial(P_i\backslash C_{\nu,i}))=0$, the Portmanteau theorem yields a weak$^*$-neighborhood $\UC_{\nu} \subset \CC$ of $\nu$ such that for every $\mu\in\UC_{\nu}$ it holds that $|\nu(P_i\backslash C_{\nu,i}) - \mu(P_i\backslash C_{\nu,i})| \leq \eps/2$. Therefore, $\mu(P_i\backslash C_{\nu,i}) \leq \eps$ for all $\mu\in\UC_{\nu}$. Since $\CC$ is weakly$^*$-compact, we can cover $\CC$ with finitely many of these neighborhoods, say $\UC_{\nu_1},\ldots,\UC_{\nu_r}$. Then $C_i := \bigcup_{i=1}^r C_{\nu_i}$ is a compact subset of $P_i$ for $1\leq i\leq k$ and for every $\mu \in \CC$ it holds that $\mu(P_i\backslash C_i) \leq \eps$, in particular for all $\mu = \mu_n$. This implies that the constant sequence $\PC_n\equiv\PC$ is in $\EC_{\Mis}$.%
\end{proof}

\begin{remark}
Note that every compact metric space admits finite measurable partitions of sets with arbitrarily small diameters and zero boundaries (cf.~\cite[Lem.~4.5.1]{KHa}).%
\end{remark}

\begin{example}
An example for systems with invariant sequences satisfying the assumption of Proposition \ref{prop_weaklystablesuffcon}, can be found in \cite{OSY}: Let $M$ be a compact connected Riemannian manifold. By $d(\cdot,\cdot)$ denote the Riemannian distance and by $m$ the Riemannian volume measure. For simplicity, we will assume that $m(M)=1$, so $m$ is a probability measure. For constants $\lambda>1$ and $\Gamma>0$ consider the set%
\begin{equation*}
  \EC(\lambda,\Gamma) := \left\{f \in \CC^2(M,M)\ :\ f \mbox{ expanding with factor } \lambda,\quad \|f\|_{\CC^2} \leq \Gamma \right\},%
\end{equation*}
where ``\emph{expanding with factor $\lambda$}'' means that $|Df_x(v)|\geq\lambda|v|$ holds for all $x\in M$ and all tangent vectors $v\in T_xM$. We will consider a NDS $f_{1,\infty} = \{f_n\}$ on $M$ with $f_n \in \EC(\lambda,\Gamma)$ for fixed $\lambda>1$ and $\Gamma>0$. It is clear that such a system is equicontinuous. We define%
\begin{equation*}
  \DC := \left\{\varphi:M\rightarrow\R\ :\ \varphi>0,\ \mbox{Lipschitz},\ \int \varphi \rmd m = 1 \right\}%
\end{equation*}
and for every $L>0$ the set%
\begin{equation*}
  \DC_L := \left\{\varphi \in \DC\ :\ \left|\frac{\varphi(x)}{\varphi(y)}-1\right| \leq L d(x,y) \mbox{ if } d(x,y)<\eps\right\},%
\end{equation*}
where $\eps>0$ is a fixed number (depending on $\lambda$ and $\Gamma$). Note that%
\begin{equation*}
  \DC = \bigcup_{L>0}\DC_L,%
\end{equation*}
since for every $\varphi\in\DC$ we have%
\begin{equation*}
  \left|\frac{\varphi(x)}{\varphi(y)}-1\right| = \frac{1}{\varphi(y)}\left|\varphi(x) - \varphi(y)\right| \leq \frac{\Lip(\varphi)}{\min \varphi}d(x,y).%
\end{equation*}
For any expanding map $f:M\rightarrow M$ we write%
\begin{equation*}
  \PC_f(\varphi)(x) = \sum_{y\in f^{-1}(x)}\frac{\varphi(y)}{|\det Df(y)|},\qquad \PC_f(\varphi):M\rightarrow\R,%
\end{equation*}
for the Perron-Frobenius operator associated with $f$ acting on densities $\varphi\in\DC$. Note that this makes sense, since the expanding maps are covering maps, and hence the sets $f^{-1}(x)$ are finite, all having the same number of elements.%

Now let $\varphi\in\DC$. We claim that the $f_{1,\infty}$-invariant sequence, defined by $\mu_1 := \varphi\rmd m$ and $\mu_n := f_1^{n-1} \mu_1$ for all $n\geq2$, has the property that the elements of the weak$^*$-closure of $\{\mu_n\}_{n\in\N}$ are pairwisely equivalent. To show this, let $L>0$ be chosen such that $\varphi\in\DC_L$ and note that $\mu_{n+1} = \PC_{f_1^n}(\varphi)\rmd m$ for all $n$. By \cite[Prop.~2.3]{OSY}, there exist $L^*>0$ and $\tau\geq 1$ such that $\PC_{f_1^n}(\varphi)\in\DC_{L^*}$ for all $n\geq\tau$. Hence, we may assume that $\PC_{f_1^n}(\varphi) \in \DC_{L^*}$ for all $n$. We will first show that the densities in $\DC_{L^*}$ are uniformly bounded away from zero and infinity and that they are equicontinuous. Assume to the contrary that there are $\varphi_n \in \DC_{L^*}$ and $x_n\in M$ such that $\varphi_n(x_n) \geq n$. Without loss of generality, we may assume that $\varphi_n(x_n) = \max_{x\in M}\varphi_n(x)$. Choosing $\delta \in (0,\eps]$ with $L\delta<1$, we obtain%
\begin{eqnarray*}
  1 &=& \int_M \varphi_n \rmd m \geq \int_{B(x_n,\delta)} \varphi_n(x) \rmd m(x) = \int_{B(x_n,\delta)} \frac{\varphi_n(x)}{\varphi_n(x_n)} \varphi_n(x_n) \rmd m(x)\allowdisplaybreaks\\
 &\geq& n \int_{B(x_n,\delta)} \left(1 - L d(x,x_n)\right) \rmd m(x)\allowdisplaybreaks\\
 &\geq& n \int_{B(x_n,\delta)} \left(1 - L\delta\right) \rmd m = n\left(1-L\delta\right) m(B(x_n,\delta)).%
\end{eqnarray*}
Since $m(B(x_n,\delta))$ is bounded away from zero, this is a contradiction. Hence, the functions in $\DC_{L^*}$ are uniformly bounded by some constant $K$. This immediately implies equicontinuity, since for $x,y\in M$ with $d(x,y)<\eps$ we have%
\begin{equation*}
  |\varphi(x)-\varphi(y)| = \varphi(y) \left|\frac{\varphi(x)}{\varphi(y)} - 1\right| \leq K L d(x,y).%
\end{equation*}
To show that the $\varphi\in\DC_{L^*}$ are uniformly bounded away from zero, assume to the contrary that there exist $\varphi_n\in\DC_{L^*}$ and $x_n\in M$ such that $\varphi_n(x_n) \rightarrow 0$. By compactness, we may assume that $x_n\rightarrow x$. Then%
\begin{equation*}
  |\varphi_n(x)-\varphi_n(x_n)| \leq KL d(x,x_n) \rightarrow 0 \qquad \Rightarrow \qquad \varphi_n(x) \rightarrow 0.%
\end{equation*}
Now pick some $y\in B(x,\eps)$. Then%
\begin{equation*}
  |\varphi_n(x) - \varphi_n(y)| = \varphi_n(x)\left|1 - \frac{\varphi_n(y)}{\varphi_n(x)}\right| \leq \varphi_n(x)L\eps \rightarrow 0.%
\end{equation*}
Using the theorem of Arzel\'a-Ascoli, we can choose a uniformly convergent subsequence $\varphi_{m_n} \rightarrow \varphi$. The above argument shows that the closed set $\varphi^{-1}(0)$ is open, and by assumption it is nonempty. Hence, it is equal to $M$. This is a contradiction to the integral condition $\int\varphi_{m_n}\rmd m = 1$, which implies $\int\varphi\rmd m = 1$. Now we prove the claim: Let $\nu$ be a weak$^*$ limit point of $\{\mu_n\}$ and let $\varphi_1 := \varphi$, $\varphi_{n+1} := \PC_{f_1^n}\varphi$. Then, for a subsequence $m_n$ and for every continuous $g:M\rightarrow\R$ we have%
\begin{equation*}
  \int_M g \varphi_{m_n} \rmd m \rightarrow \int_M g \rmd \nu.%
\end{equation*}
On the other hand, by the theorem of Arzel\'a-Ascoli, we may assume that $\varphi_{m_n}$ converges uniformly to some $\varphi^*$, which is bounded away from zero and satisfies $\int \varphi^* \rmd m = 1$. Hence,%
\begin{equation*}
  \int_M g \varphi_{m_n} \rmd m \rightarrow \int_M g \varphi^* \rmd m,%
\end{equation*}
implying $\nu = \varphi^* \rmd m$.%
\end{example}

\begin{remark}
The above example can be regarded as a nonautonomous version of the classical result of Krzyzewski and Szlenk \cite{KSz} which asserts that every expanding $\CC^2$-map has an absolutely continuous invariant measure.%
\end{remark} 

\begin{remark}
In view of Proposition \ref{prop_weaklystablesuffcon} and Proposition \ref{prop_emisequiconj}, the most general criterion which guarantees a large Misiurewicz class for an equicontinuous system $(X_{1,\infty},f_{1,\infty})$ with invariant sequence $\mu_{1,\infty}$ is the existence of an equiconjugacy to a system which satisfies the assumptions of Proposition \ref{prop_weaklystablesuffcon}. That is, there exists a compact metric space $X$ and an equicontinuous sequence $\{\pi_n\}$ of homeomorphisms $\pi_n:X_n \rightarrow X$ such that all elements of the weak$^*$-closure of the set $\{\pi_n\mu_n\}$ are equivalent.%
\end{remark}

\section{Concluding Remarks and Open Questions}%

In this paper, we introduced a notion of metric entropy for quite general nonautonomous dynamical systems and studied its elementary properties, in particular its relation to the topological entropy defined by Kolyada, Misiurewicz, and Snoha. The number of open questions about this new quantity tends to infinity. We restrict ourselves to a very short list of questions and topics for future research:%
\begin{itemize}
\item In order to obtain a fruitful theory of metric entropy for nonautonomous systems, it seems inevitable to find appropriate analogues of the notion of ergodicity. Describing ergodicity as the property that the state space cannot be broken apart into two invariant subsets of positive measure, one can use the same definition for a metric NDS on a single probability space. However, this definition is probably too strict. It seems more likely that for different purposes different analogues of ergodicity of varying strength will fit.%
\item One of the next steps in the further development of the entropy theory for nonautonomous systems certainly is the study of the question to which extent the variational inequality (Theorem \ref{thm_general_varinequality}) can be extended to a full variational principle. Another interesting question is under which conditions there exist reasonably small generating sets for the Misiurewicz class.%
\item The classical Pesin formula and Margulis-Ruelle inequality relate the metric entropy of a diffeomorphism to its Lyapunov exponents, given by the Multiplicative Ergodic Theorem. It is an interesting and probably very far-reaching question to which extent such results can be transferred to the nonautonomous case.%
\item The notion of metric entropy in this paper also generalizes the metric sequence entropy introduced in Kushnirenko \cite{Kus}. It might be an interesting topic for future research to look for generalizations of the known results about metric sequence entropy.%
\end{itemize}

\small{

}

\end{document}